\documentclass[english, 12pt]{amsart}
\usepackage{nicefrac}
\usepackage[inline]{enumitem}
\usepackage{enumitem}
\usepackage{todonotes}
\usepackage{amssymb}
\usepackage{mathrsfs}
\definecolor{myRed}{rgb}{0.9,0.,.2}
\definecolor{myBlue}{rgb}{0.,0.,.6}
\definecolor{myGreen}{rgb}{0.1,0.7,0.1}
\definecolor{myViolet}{rgb}{102,0,153}
\usepackage[colorlinks=true,urlcolor=myViolet,linkcolor=myBlue,citecolor=myGreen,pagebackref=true]{hyperref}

\usepackage{fouriernc}%%%%%%%%%%%%%%
\usepackage[scaled=0.775]{helvet}
\usepackage{courier}
\usepackage[marginparwidth=2cm]{geometry}
\usepackage{array,float}

\usepackage{tikz}
\usetikzlibrary{shapes,arrows}
\usetikzlibrary{fit,positioning}
\usetikzlibrary{patterns,decorations.pathreplacing}
\usepackage{xkeyval}
\usepackage{moreverb}
\usepackage{epic}
\usepgfmodule{shapes,plot,decorations}
\usepackage{booktabs} 
\usepackage[normalem]{ulem}

\newtheorem{theoremA}{\indent Theorem}
\newtheorem{theoremB}{\indent Theorem}
\newtheorem{theoremC}{\indent Theorem}
\renewenvironment{proof}
	{\par\indent{\bf Proof.}}
	{\hfill$\scriptstyle\blacksquare$}

\makeatletter
  \def\section{\@startsection{section}{2}%
    {\z@}{.5\linespacing\@plus.7\linespacing}{.5em}%
    {\normalfont\bfseries\centering}}
\def\@secnumfont{\bfseries}
\makeatother

\newcommand{\Z}{\mathbb{Z}}
\newcommand{\Q}{\mathbb{Q}}

\newcommand{\R}{\mathbb{R}}

\newcommand{\E}{\mathbb{E}}

\newcommand{\SSS}{\mathbb{S}}

\newcommand{\HH}{\mathbb{H}}
\newcommand{\F}{\mathbb{F}}

\newcommand{\ta}{\mathbf{t_{\overline{\alpha}}}}

\newcommand{\OO}{\mathrm{O}}
\newcommand{\OOO}{\mathcal{O} }

\newcommand{\id}{\mathrm{id} }

\DeclareMathOperator*{\Isom}{Isom}

\DeclareMathOperator*{\Sym}{Sym}

\newcounter{z}

\newtheorem{theorem}{\indent Theorem}[section]
\newtheorem{lemma}{\indent Lemma}[section]
\newtheorem{corollary}{\indent Corollary}%[section]%[theorem]

\newtheorem{proposition}{\indent Proposition}[section]
\newtheorem{definition}{\indent Definition}[section]

\theoremstyle{definition}
\newtheorem{remark}{\indent Remark}%[section]

\newcommand{\arcsinh}{\mathop{\rm arcsinh}\nolimits}

\begin{document}

\title{From geometry to arithmetic of compact hyperbolic Coxeter polytopes}

\author{Nikolay Bogachev}

\address{The Kharkevich Institute for Information Transmission Problems, Moscow, Russia}
\address{Moscow Institute of Physics and Technology, Dolgoprudny, Russia}
\address{Skolkovo Institute of Science and Technology, Skolkovo, Russia}
\address{Caucasus Mathematical Center, Adyghe State University, Maikop, Russia}

\email[]{nvbogach@mail.ru}

\begin{abstract}
We establish some geometric constraints on compact Coxeter polytopes in hyperbolic spaces and show that these constraints can be a very useful tool for the classification problem of reflective anisotropic Lorentzian lattices and cocompact arithmetic hyperbolic reflection groups.

\end{abstract}

\maketitle

\begin{center}
  \textit{Dedicated to the memory of Ernest Borisovich Vinberg (1937 -- 2020)}
\end{center}

\vskip 0.7cm

%\tableofcontents

\section{Introduction}
\label{s1}

The main purpose of this paper is two-fold: to establish geometric constraints on compact Coxeter polytopes in hyperbolic spaces (Theorems~\ref{th1}~\&~\ref{th2}) and to show that these constraints can be a very useful tool for the classification problem of reflective anisotropic Lorentzian lattices and cocompact arithmetic hyperbolic reflection groups (Theorem~\ref{th3}).

The classification of reflective Lorentzian lattices is a very hard problem even for a fixed dimension and a ground field. This problem remains open in general since 1970--80s, and it is especially difficult for anisotropic lattices because of the lack of efficient methods, while for isotropic ones there is a general approach due to Scharlau. A significant progress in the classification of anisotropic reflective lattices was achieved only in the $2$-dimensional case and the only successful approach was geometric due to Nikulin.

Recall that Coxeter polytopes (i.e., polytopes whose bounding
hyperplanes $H_i$ and $H_j$ either do not intersect or form a dihedral angle of $\pi/n_{ij}$, 
where $n_{ij}\in\Z$, $n_{ij} \geqslant 2$) are fundamental domains for discrete groups generated by reflections in hyperplanes in spaces of constant curvature.
Finite volume Coxeter polytopes in $\E^n$ and
$\SSS^n$ were classified by Coxeter himself
in 1933 \cite{Cox34}.
In 1967, Vinberg \cite{Vin67} initiated his theory of hyperbolic reflection groups and, in particular, proved an arithmeticity
criterion for finite covolume hyperbolic reflection groups.
Due to the impressive results of Vinberg, Nikulin, Agol, Belolipetsky and others it is known (see \cite{Vin84,Nik07,ABSW08}) that there are only finitely many maximal  
arithmetic hyperbolic reflection groups in all  dimensions $n \geqslant 2$ and they can exist in $\mathbb{H}^n$ only for $n < 30$.
These results give the hope that reflective Lorentzian lattices and maximal arithmetic hyperbolic reflection 
groups can be classified.
For the detailed discussion and precise definitions of arithmetic hyperbolic reflection groups and reflective Lorentzian lattices see \textbf{\S}~\ref{sec:arith}.

Note that hyperbolic Coxeter polytopes belong to the class of \emph{acute-angled} (i.e., with dihedral angles at most $\pi/2$) polytopes in $\HH^n$. It is worth mentioning that in an acute-angled polytope, the distance from an interior point to a face (of any dimension) is equal to the distance to the plane of this face, as well as the distance between two \textit{facets} (i.e.,  faces of codimension~$1$) is equal to the distance between the corresponding supporting hyperplanes.

In order to formulate the main results of our paper, we introduce some notation. Let $P$ be a compact acute-angled polytope, $E$ an edge of $P$, $F_1, \ldots, F_{n-1}$ the
facets of $P$ containing $E$, and let $F_n$ and $F_{n+1}$ be the \emph{framing} \emph{facets} of $E$, i.e., the facets containing vertices of $E$ but not $E$ itself. 
The collection $\Sigma_E$ of facets $F_1, \ldots, F_{n+1}$ is called the \emph{ridge associated with} $E$ and
the number $\cosh \rho(F_n, F_{n+1})$ is its \emph{width}, where $\rho{(\,\cdot\,, \cdot\,)}$ is the hyperbolic metric and $\rho(F_n, F_{n+1})$ is the distance between the facets $F_n$ and $F_{n+1}$.

Every ridge corresponds to a set $\overline{\alpha} = \{\alpha_{ij}\}$, where $\alpha_{ij}$ is the dihedral angle between the facets $F_i$ and $F_j$. In \cite{Kol12}, the ridge $\Sigma_E$ associated to the edge $E$ was called the ridge of type $\overline{\alpha}$. We denote by $\Omega$ the set of all possible sets (or types) $\overline{\alpha}$.

A Lorentzian lattice $L$ is said to be
\emph{reflective}
if its automorphism group contains a finite index subgroup
generated by
reflections, and
\emph{sub-$2$-reflective}
if the same group contains a finite index subgroup generated by
\emph{sub-$2$-reflections} (see Definitions~\ref{def:sub-2} and \ref{def:refl-lat}).

In order to classify reflective Lorentzian lattices and
to prove his finiteness theorems for maximal arithmetic hyperbolic
reflection groups,
Nikulin proved
(see \cite[Lemma~3.2.1]{Nik80}; the proof
of \cite[Theorem~4.1.1]{Nik81b}; or his ICM 1986 talk \cite[Theorem A]{Nik86}) that every
finite volume
acute-angled polytope in~$\mathbb{H}^n$ has a facet $F$ such that
$\cosh \rho (F_1, F_2) \leqslant 7$
for any facets $F_1$ and $F_2$ of $P$ adjacent to $F$. 
This
implies that every compact (even finite volume)
acute-angled polytope
$P \subset \HH^n$ contains a ridge of width $\leqslant 7$ (for the case $n=3$ see \cite[Prop.~2.1]{Bog19}).
Note that we present this assertion in a form convenient  for us, although in most of Nikulin's papers (excepting \cite[Theorem A]{Nik86}) it was formulated in a different way. More precisely, in his papers, the squared lengths  of the facet
normals are equal to $(-2)$, therefore, his bound appears in the form $(\delta, \delta') \leqslant  14$. 

The next theorem is the first main result of this paper.

\begin{theoremA}\label{th1}
Every compact Coxeter polytope in the hyperbolic
$3$-space $\HH^3$ contains a ridge of width less than some number $\ta$ which depends  only on the set $\overline\alpha$ (and does not depend on the whole polytope) and can be explicitly computed (see Prop.~\ref{th:w}). 

Moreover,
$$
\mathbf{t}_{(\pi/k, \alpha_{13}, \alpha_{14}, \alpha_{23}, \alpha_{24})} < 5 \ \text{for} \ 2 \leqslant k \leqslant 4, \qquad \mathbf{t}_{(\pi/k, \alpha_{13}, \alpha_{14}, \alpha_{23}, \alpha_{24})} < 3 \ \text{for} \ k \geqslant 6,
%\mathbf{t}_{(\pi/2, \alpha_{13}, \alpha_{14}, \alpha_{23}, \alpha_{24})} < 5, \quad \ta < 5 \ \text{if} \ \alpha_{12} \ne \frac{\pi}{5},
$$
and finally $\alpha_{12} = \pi/5$ corresponds to the maximal upper bound:
$$
\max_{\overline \alpha \in \Omega} \{\ta\}=\mathbf{t}_{(\pi/5, \pi/3, \pi/3, \pi/2, \pi/2)} < 5.75.
$$
\end{theoremA}

\begin{remark}\label{rem-ta-Q}
This new bound $\ta$ is more efficient than
Nikulin's estimate. For certain types $\overline \alpha$ the corresponding number $\ta$ is significantly less than $7$, and even less than $5$ (see e.g. Table~\ref{tab:ridges}).

The author used such bounds $\ta$ (see
\cite{Bog19}) for classifying sub-$2$-reflective Lorentzian lattices over $\Z$ of signature $(3,1)$. Using Nukulin's result, one gets around $280$ candidate lattices to be combed through and checked for reflectivity, while using $\ta$ leaves us with at most $85$ candidates.

Moreover, in order to compute $\ta$ for specific types $\overline \alpha$ and to obtain a list of candidate lattices (over a fixed ground field) the author implemented the program {\bf \texttt{SmaRBA}} (\textbf{Sma}ll \textbf{R}idges, \textbf{B}ounds and \textbf{A}pplications, see \cite{SmaRBA}).
\end{remark}

\medskip
\begin{remark}
In a recent paper by the author, it has been shown that every compact arithmetic Coxeter polytope in $\HH^3$ with the ground field
$\Q$ contains a ridge of width $< 4.14$, see \cite[Theorem 1.1]{Bog19}, however, due to a minor technical error, the correct
bound there should be $4.98$. We shall discuss it in \textbf{\S} \ref{section:appl}. Observe that Theorem~\ref{th1} is much more general, since
arithmetic Coxeter polytopes in $\HH^3$ with ground field $\Q$ can have the dihedral angles
$\pi/2$, $\pi/3$, $\pi/4$, and $\pi/6$ only.
\end{remark}

\begin{remark}\label{rem:degrees}
The general upper bounds on $\ta$ in Theorem~\ref{th1} can be very useful for obtaining upper bounds on degrees of ground fields of arithmetic hyperbolic reflection groups. We conjecture that Theorem~\ref{th1} can be generalized to higher dimensions and that actually this upper bound will decrease with increasing dimension of the hyperbolic space, see Theorem~\ref{th2} for a particular case.

In his short and bright paper \cite{Nik07} Nikulin proved that the degree of the ground field of an arithmetic hyperbolic reflection group is bounded above by the maximum of the degrees of ground fields of arithmetic reflection groups in $\HH^n$, for $n = 2, 3$, and some number $N(2\mathbf{t})$ which depends on
his universal estimate $\mathbf{t} = 7$.
In 2011, Nikulin \cite{Nik11} proved that 
$N(14) \leqslant 25$. It was proved by Linowitz \cite{Lin18} in 2018 that $d \leqslant 7$ for $n=2$. In 2014, Belolipetsky and Linowitz showed \cite{BL2014} that  $d \leqslant 9$ in $\HH^3$.
Thus, the degree of the ground field of an arithmetic hyperbolic reflection group is bounded above by $25$ in dimensions $n \geqslant 4$. 

Nikulin's proof and construction explicitly use his estimate $\mathbf{t} < 7$ (more precisely, the doubled constant $2 \cdot 7 = 14$). This constant appeared in Nukilin's paper as the so-called ``minimality of edged polyhedra'', which is similar to the width of small ridges in our case.

Therefore, having a better bound for $\ta$ in $\HH^{\geqslant 4}$, one can improve (perhaps, significantly) Nikulin's bounds for degrees of ground fields. The author tried to use Nikulin's methods with $\ta = 5.75$ (assuming that our conjecture is already proved at least in its weakest sense) and obtained that the degree $d$ of ground field should be $\leqslant 23$ instead of $25$. If the angle $\pi/5$ is not allowed in arithmetic Coxeter polytopes (this condition is not very restrictive actually), then we use $\ta = 5$ and obtain $d \leqslant 21$. 
\end{remark}

The next easy consequence of Theorem~\ref{th1} can be also considered in the context of the above remark.

\begin{corollary}
\label{corol:th1}
Let $P \subset \HH^{n \geqslant 4}$ be a compact Coxeter
polytope and let $P'$ be a $3$-dimensional face of
$P$ that is a Coxeter polytope itself.
Then $P$ has
a ridge of width
$< 5.75$.
\end{corollary}

\begin{remark}
It is proved in \cite{BK20} that each face of a quasi-arithmetic Coxeter polytope that is itself a Coxeter polytope is  also quasi-arithmetic; in addition,  
a sufficient condition for a face of codimension $1$  to be arithmetic is provided. 
A large number of Coxeter polytopes and their faces has been studied by using a computer program {\bf \texttt{PLoF}} \cite{plof}. 
It turns out that it is a common situation that a Coxeter polytope has many faces that are also Coxeter polytopes. 
This means that the  condition in Corollary~\ref{corol:th1} is rather natural. 
\end{remark}

\vskip 0.2cm

We shall say that a ridge $\Sigma_E$ of a compact hyperbolic acute-angled polytope $P$ is \emph{right-angled} if $\alpha_{ij} = \pi/2$ whenever $1 \leqslant i < j \leqslant n+1$, where $\alpha_{ij}$ is the dihedral angle between the facets $F_i$ and $F_j$. An edge $E$ of a polytope $P$ is said to be {\em outermost} relative to $O$ if 
$$\rho(O, E) = \max_{E'} \rho(O,E'),$$
where the maximum is taken over all edges $E'$ of $P$. It is possible that there are several outermost edges, in this case we fix one of them. The following result can be considered as a particular confirmation of a conjecture discussed in the last part of Remark~\ref{rem:degrees}.

\begin{theoremB}\label{th2}
Let $P$ be a compact Coxeter polytope in $\HH^{n \geqslant 3}$, let $O$ be an interior point of $P$, and let $E$ be the outermost edge from $O$. If the ridge associated with $E$ is right-angled, then it has width $\mathbf{t}(n) < \frac{n+1}{n-1}$. In particular, any compact right-angled Coxeter polytope has a ridge of width $\mathbf{t}(n) < \frac{n+1}{n-1}$. 
\end{theoremB}

\begin{remark}
It is easy to see that $\mathbf{t}(n) > 1$ (i.e., in the right-angled case the framing facets can not intersect even at infinity) and $\lim_{n \to \infty} \mathbf{t}(n) = 1$.

The ridge in Theorem~\ref{th1} also corresponds to the outermost edge.
\end{remark}

\vskip 0.2cm

%\vskip 0.3cm

In order to formulate the third main result of our paper, we introduce some notation:

\begin{itemize}
\item[1)]~$[C]$ is a quadratic lattice whose inner product in some
basis is given by a symmetric matrix $C$;

\item[2)]~$d(L) := \det C$ is the discriminant of the lattice $L = [C]$;

\item[3)]~$L \oplus M$ is the orthogonal sum of the lattices $L$ and $M$;

\item[4)]~$\OOO'(L)$ is the automorphism group of $L$ preserving $\HH^n$;

\item[5)]~$\OOO_r (L)$ is the subgroup of $\OOO'(L)$ generated by all
reflections contained in it.
\end{itemize}

Recall (see \cite[Prop. 3]{Vin72}) that $\OOO'(L) = \OOO_r(L) \rtimes H,$ where $H = \Sym(P) \cap \OOO'(L)$ and $P$ is the fundamental Coxeter polytope of the arithmetic hyperbolic reflection group $\OOO_r(L)$.

\begin{theoremC}\label{th3}
Every maximal sub-$2$-reflective  Lorentzian lattice
$L$ of signature $(3,1)$ over $\Z[\!\sqrt{2}]$ is isomorphic
to one in the following list:
\vskip 0.5 cm
%\begin{small}
\begin{center}
\begin{tabular}{c|c|c|c}
No. & $L$ & $\#$ facets & $d(L)$\\
\hline
$1$ & $[-1 - \sqrt{2}] \oplus [1] \oplus [1] \oplus [1]$ & $5$ & $-1 - \sqrt{2}$\\
\hline
$2$ & $[-1 - 2\sqrt{2}] \oplus [1] \oplus [1] \oplus [1]$ & $6$ & $-1 - 2\sqrt{2}$\\
\hline
$3$ & $[-5 - 4\sqrt{2}] \oplus [1] \oplus [1] \oplus [1]$ & $5$ & $-5 - 4\sqrt{2}$\\
\hline
$4$ & $[-11 - 8\sqrt{2}] \oplus [1] \oplus [1] \oplus [1]$ & $17$ & $-11 - 8\sqrt{2}$\\
\hline
$5$ & $[-\sqrt{2}] \oplus [1] \oplus [1] \oplus [1]$ & $6$ & $-\sqrt{2}$\\
\hline
$6$ & $[-7 - 5\sqrt{2}] \oplus [1] \oplus [1] \oplus [1]$ & $5$ & $-7 - 5\sqrt{2}$\\
\hline
$7$ &  $\begin{bmatrix}
2 & 0 & 0 & -1\\
0 & 2& -1 & -1\\
0 & -1 & 2 & -\sqrt{2}\\
-1 & -1 &  -\sqrt{2} & 2
\end{bmatrix}$ & $5$ & $-3-4\sqrt{2}$ \\
\end{tabular}
\end{center}
\vskip 0.2 cm
(Here, ``$\#$ facets'' denotes the number of facets of
the fundamental Coxeter polytope $P$ for the maximal
arithmetic
hyperbolic reflection subgroup $\OOO_r (L)$,
preserving
$L$.)

Moreover, these seven lattices are not isomorphic to each other and correspond to seven different automorphism groups. 

The Coxeter --- Vinberg diagrams of the fundamental Coxeter polytopes for the corresponding maximal arithmetic hyperbolic reflection groups $\OOO_r (L)$ are depicted in Fig.~\ref{fig:thC} excepting the group $\OOO_r(no.~4)$ whose diagram can be found in {\bf \texttt{SmaRBA}} \cite{SmaRBA}.
\end{theoremC}

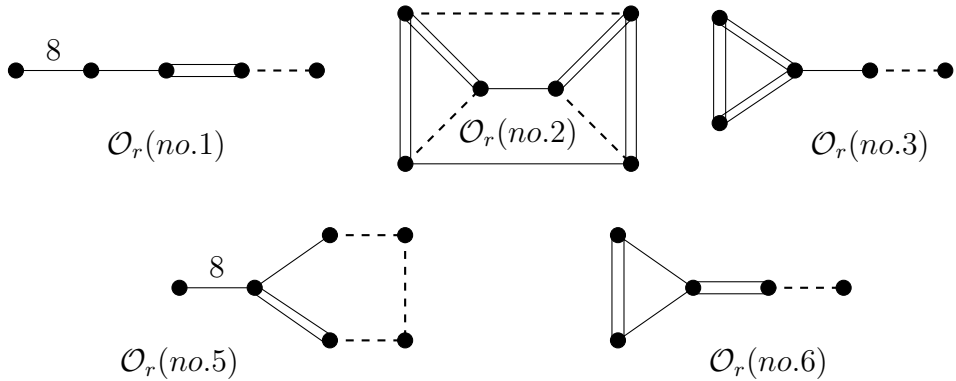
\begin{figure}
\begin{center}
\begin{tikzpicture}
\draw[fill=black]
(0,0)   circle [radius=.1]  
(1,0)   circle [radius=.1]  
(2,0)   circle [radius=.1] node at (2,-1) {$\OOO_r(no.1)$}
(3,0)   circle [radius=.1] 
(4,0)   circle [radius=.1] 

%node [below] {$F_2$}
% t ($(0.97, 0.92)!0.25!(0.97,-0.92)$)  {$\mathbf{2}$} ;    
;
\draw

(0,0) -- (1,0) node [midway,above] {$8$}
(1,0) -- (2,0)
(2,0.08) -- (3,0.08)
(2,-0.08) -- (3,-0.08);

\draw[thick, dashed] (3,0) -- (4,0)   %node [above] {$8$}
;
\end{tikzpicture}
\qquad \qquad  
\begin{tikzpicture}
\draw[fill=black]
(0,0)   circle [radius=.1]  

(1,0)   circle [radius=.1]  

(2,1)   circle [radius=.1]  

(2,-1)   circle [radius=.1] 

(-1,1)   circle [radius=.1]  

(-1,-1)   circle [radius=.1] 

;
\draw

(-1.07,-1) -- (-1.07,1)
(-0.93,-1) -- (-0.93,1)

(2.07,-1) -- (2.07,1)
(1.93,-1) -- (1.93,1)

(0,0.08) -- (-0.92,1)
(-0.08,0) -- (-1,0.92)

(1.08,0) -- (2,0.92)
(1,0.08) -- (1.92,1)

(0,0) -- (1,0) node at (0.5,-0.6) {$\OOO_r(no.2)$}
(-1,-1) -- (2,-1)

 ;

\draw[thick, dashed] (1,0)  -- (2,-1);
 
\draw[thick, dashed] (-1,-1)  -- (0,0);
 
\draw[thick, dashed] (-1,1) -- (2,1);

\end{tikzpicture}
\qquad \qquad  
\begin{tikzpicture}
\draw[fill=black]
(0,0.7)   circle [radius=.1]  
(0,-0.7)   circle [radius=.1]  
(1,0)   circle [radius=.1]  
(2,0)   circle [radius=.1] node at (2,-1) {$\OOO_r(no.3)$}
(3,0)   circle [radius=.1]

%node [below] {$F_2$}

;
\draw

(0.08,-0.7) -- (1,-0.08) 
(0,-0.62) -- (0.92,0)

(0,0.62) -- (0.92,0)
(0.08,0.7) -- (1,0.08)

(1,0) -- (2,0)

(0.08,-0.7) -- (0.08,0.7)
(-0.08,-0.7) -- (-0.08,0.7)
;

\draw[thick, dashed] (2,0) -- (3,0)   %node [above] {$8$}
;
\end{tikzpicture}
\vskip 0.7cm

\begin{tikzpicture}
\draw[fill=black]
(0,0)   circle [radius=.1]  
node at (0,-1) {$\OOO_r(no.5)$}
(1,0)   circle [radius=.1]  
(2,0.7)   circle [radius=.1] 
(2,-0.7)   circle [radius=.1] 
(3,0.7)   circle [radius=.1] 
(3,-0.7)   circle [radius=.1] 
%node [below] {$F_2$}

;
\draw

(0,0) -- (1,0) node [midway,above] {$8$}
(1,0) -- (2,0.7)
(1,-0.08) -- (2,-0.78)
(1.07,0) -- (2,-0.62)
;

\draw[thick, dashed] (2,0.7) -- (3,0.7) ; 
\draw[thick, dashed] (2,-0.7) -- (3,-0.7)  ;
\draw[thick, dashed] (3,0.7) -- (3,-0.7)  
%node [above] {$8$}
;
\end{tikzpicture} 
\qquad  \qquad  
\begin{tikzpicture}
\draw[fill=black]
(0,0.7)   circle [radius=.1]  
(0,-0.7)   circle [radius=.1]  
(1,0)   circle [radius=.1]  
(2,0)   circle [radius=.1] 
node at (2,-1) {$\OOO_r(no.6)$}
(3,0)   circle [radius=.1] 

%node [below] {$F_2$}

;
\draw

(0,-0.7) -- (1,0) 
(0,0.7) -- (1,0)
(1,0.08) -- (2,0.08)
(1,-0.08) -- (2,-0.08)
(0.08,-0.7) -- (0.08,0.7)
(-0.08,-0.7) -- (-0.08,0.7)
;

\draw[thick, dashed] (2,0) -- (3,0)   %node [above] {$8$}
;
\end{tikzpicture}
\qquad \qquad  
\begin{tikzpicture}
\draw[fill=black]
(0,0.7)   circle [radius=.1]  
(0,-0.7)   circle [radius=.1] 
(1,0.7)   circle [radius=.1]  
(1,-0.7)   circle [radius=.1]  
(2,0)   circle [radius=.1]  

node at (2,-1) {$\OOO_r(no.7)$}

%node [below] {$F_2$}

;
\draw

(0,0.7) -- (1,0.7)

(0,-0.7) -- (1,-0.7)

(1,-0.7) -- (2,0) 
(1,0.7) -- (2,0)

(1.08,-0.7) -- (1.08,0.7)
(0.92,-0.7) -- (0.92,0.7)
;

\draw[thick, dashed] (0,0.7) -- (0,-0.7)   %node [above] {$8$}
;
\end{tikzpicture}
\end{center}
\caption{The Coxeter --- Vinberg diagrams for lattices no.1--no.3 \& no.5--no.7}
\label{fig:thC}
\end{figure}

\begin{remark}
The program {\bf \texttt{SmaRBA}} obtained $83$ candidate $\Z[\!\sqrt{2}]$-lattices (up to finite extensions and pairwise isomorphisms) for sub-$2$-reflectivity via the bound $\ta$ from Theorem~\ref{th1}, while using Nikulin's estimate leaves us around $160$ lattices (again up to finite extensions and pairwise isomorphisms). One can see that the difference between the $\ta$-method and Niku\-lin's one is not as big as for the ground field $\Q$ (cf. Rem.~\ref{rem-ta-Q}). The author believes that this can be explained by the  admissibility condition (it is very restrictive) of Lorentzian quadratic forms over $\Q(\!\sqrt{2})$.
\end{remark}
\vskip 0.2cm

The author hopes that analysing small ridges can become a useful tool for classifying not only
sub-$2$-reflective Lorentzian lattices, but reflective lattices in general.

\subsection*{Organization of the paper}
The paper is organized as follows.  In
\textbf{\S}~\ref{sec:prelim} we provide  some preliminary results. Then,
\textbf{\S}~\ref{section:th1} is devoted to the proof of  Theorem~\ref{th1} (the proof of Corollary~\ref{corol:th1} is presented in \textbf{\S}~\ref{subsect:corolth1}) and
\textbf{\S}~\ref{section:th2} is devoted to the proof of  Theorem~\ref{th2}.

The proof of  Theorem~\ref{th1} is based on
Theorem~\ref{th:lenedge} (where an explicit upper bound for the length of the outermost edge of a compact acute-angled polytope in $\HH^3$ is obtained) and Proposition~\ref{th:w} (with an explicit formula for $\ta$). A more detailed plan of the proof of  Theorem~\ref{th1} is described in
\textbf{\S}~\ref{subsection:th1}.

Some definitions and facts concerning arithmetic hyperbolic reflection groups and reflective Lorentzian lattices are collected in
\textbf{\S}~\ref{sec:arith}. Finally,
\textbf{\S}~\ref{section:appl} is a description of applications of  Theorem~\ref{th1} to
classification of sub-$2$-reflective Lorentzian lattices and \textbf{\S}~\ref{section:th3} contains the proof of  Theorem~\ref{th3}.

\subsection*{Acknowledgments}

The author is grateful to Daniel Allcock and Alexander Kolpakov
for valuable discussions, helpful remarks and suggestions, to Stepan Alexandrov for remarks and corrections, and to Markus Kirschmer for his help with lattices over number fields.
The author is also
thankful to Institut des Hautes \'Etudes Scientifiques
--- IHES, and especially to Fanny Kassel,
for their
hospitality while this work was carried out.
Finally, the author thanks two anonymous referees and the editor Mikhail Kapovich for many constructive remarks and suggestions, which helped to significantly improve the paper.

\subsection*{Funding}
This work was supported by the RFBR grant 18-31-00427.
A software implementation {\bf \texttt{VinAl}} \cite{VinAlg2017} used in this paper for testing a reflectivity of candidate-lattices (see \textbf{\S~\ref{section:th3}}) was supported by the Russian Science Foundation (project no. 18-71-00153).

\subsection*{Dedication to \`Ernest Borisovich Vinberg (26.07.1937 -- 12.05.2020)}

This paper is dedicated to \`Ernest B. Vinberg, who was my scientific advisor. He is known for his fundamental results and breakthrough discoveries in discrete subgroups of Lie groups, Lie groups and algebras, as well as in invariant theory, representation theory, and algebraic geometry. 

One of his most beautiful discoveries is the theory of hyperbolic reflection groups initiated in 1967. There Vinberg described fundamental domains (hyperbolic Coxeter polytopes) for such groups in terms of their Gram matrices and Coxeter --- Vinberg diagrams. He also provided an arithmeticity criterion for hyperbolic reflection groups of finite covolume. In 1972, Vinberg developed an algorithm that is now widely used for constructing fundamental polytopes of hyperbolic reflection groups. In 1981, he obtained the following celebrated and surprising result: there are no compact hyperbolic Coxeter polytopes and no arithmetic finite volume Coxeter polytopes in $\HH^{\geqslant 30}$. In 1983, Vinberg was an Invited Speaker at the International Congress of Mathematicians. In 2014, he constructed the first examples of higher-dimensional non-arithmetic non-compact hyperbolic Coxeter polytopes.
 
Vinberg was a great mathematician who liked finding symmetries and transformation groups in the realm of mathematics and beyond, thus giving this concept some more philosophical meaning. He created a large mathematical school with many students: some of them have become famous mathematicians.
 
To this day, I am very grateful to \`Ernest B. Vinberg, whose encouragement, constant help and invaluable advice were so important for me over the years.

\section{Preliminaries}\label{sec:prelim}

\subsection{Hyperbolic Lobachevsky space and convex polytopes}

Let $\R^{n,1}$ be the $(n+1)$-dimensional
pseudo-Euclidean real
\emph{Minkowski space}
equipped with the inner product
$$
(x, y) = - x_0 y_0 + x_1 y_1 + \ldots + x_n y_n
$$
of signature $(n,1)$.
A \emph{vector model} of the \emph{$n$-dimensional hyperbolic Lobachevsky space} $\HH^n$ is the above component of the standard hyperboloid lying in the \emph{future light
cone}:
$$
\HH^n = \{x \in \R^{n,1} \mid (x,x) = -1,\ x_0 > 0\}.
$$
The points of $\HH^n$ are called
 \emph{proper points}.
The \emph{points at infinity} (or on \emph{the boundary} $\partial
\HH^n$) in this model correspond to \emph{isotropic
one-dimensional subspaces} of $\R^{n,1}$, that is, vectors $x \in \R^{n,1}$ such that $(x,x) = 0$.

The hyperbolic metric $\rho$ is given by
$$
\cosh \rho (x, y) = - (x, y).
$$

Let $\OO_{n,1} (\R)$ be the group of
orthogonal transformations of the space $\R^{n,1}$, and let $\OO^+_{n,1}(\R)$ be
its subgroup of index $2$ preserving $\HH^n$.
The group
$\OO^+_{n,1}(\R) \simeq \Isom(\HH^n)$ is the \emph{isometry
group} of the hyperbolic $n$-space $\HH^n$.

Suppose that $e \in \R^{n,1}$ is a unit vector (that is, $(e,e) = 1$). Then the set
$$
H_{e} = \{x \in \HH^n \mid (x,e) = 0\}
$$
is a \emph{hyperplane} in $\HH^n$ and it divides the entire space into the \emph{half-spaces}
$$
H^-_e = \{x \in \HH^n \mid (x,e) \leqslant  0\}, \qquad H^+_e = \{x \in \HH^n \mid (x,e) \geqslant 0\}.
$$
The orthogonal transformation given by the formula
$$
\mathcal{R}_e (x) = x - 2 (e, x) e,
$$
is called the \emph{reflection in the hyperplane}
$H_e$, which
is called the \emph{mirror} of $\mathcal{R}_e$.

\begin{definition}
A convex polytope in $\HH^n$ is the intersection of finitely many half-spaces that has non-empty interior. 
A generalized convex polyhedron is the intersection (with non-empty interior) of a family (possibly, infinite) 
of half-spaces such that every ball intersects only finitely many of their boundary hyperplanes.
\end{definition}

\begin{definition}
A generalized convex polyhedron is said to be acute-angled if all its dihedral angles do not exceed $\pi/2$. A generalized convex polyhedron is called a Coxeter polyhedron if all its dihedral angles are of the form $\pi/k$, where $k \in \{2,3,4,\ldots,+\infty\}$.
\end{definition}

It is known that the fundamental domains of discrete reflection groups are generalized Coxeter polyhedra (see \cite{Vin67, Vin85, VS88}).

A convex polytope has finite volume if and only if
it is equal to the
convex hull of finitely many points of the closure $\overline{\HH^n} = \HH^n \cup \partial \HH^n$. If a polytope is compact,
 then it is a convex hull of finitely many proper points of $\HH^n$.

It is also known \cite{Vin85, AVS88} that compact acute-angled polytopes, in particular, compact Coxeter polytopes in
$\HH^n$, are \emph{simple}, that is, every vertex
belongs to exactly $n$ facets (and $n$ edges). 

\subsection{Bounds for the length of the outermost edge
for a compact acute-angled polytope in $\mathbb{H}^3$}
\label{ss2.3}

In this subsection, $P$ denotes a compact acute-angled polytope in the three-dimensional
Lobachevsky space $\mathbb{H}^3$.
Following Nikulin \cite[Theorem 4.1.1]{Nik81b}, we consider an interior point $O$ in $P$. Let $E$ be
\textit{the outermost
edge from it} and
$V_1$ and $V_2$ be the vertices of $E$.
Recall that in an acute-angled polytope,
the distance from an interior point to a
face (of any dimension) is equal to the
distance to the plane of this face.

Let $F_1$ and $F_2$ be the facets of $P$ containing the edge $E$ and let $F_3$ and $F_4$ be the framing facets of $E$.
Let $E_1$ and $E_3$ be the edges of the polytope $P$ outgoing
from the vertex $V_1$ and let $E_2$ and $E_4$ be the edges outgoing from $V_2$ such that the edges
$E_1$ and $E_2$ lie in the face $F_1$. The length of the edge $E$ is
denoted by $a$, and the plane angles between the edges
$E_j$ and $ E $ are denoted by $\alpha_j$ (see Fig.~\ref{fig1}).

Denote by $V_1 I$, $V_2 I$, $V_1 J$, $V_2 J$ the bisectors of angles $\alpha_1$,
$\alpha_2$, $\alpha_3$, $\alpha_4$, respectively. Let $h_I$ and $h_J$ be
the distances from the points $I$ and $J$ to the edge $E$.

\begin{figure}[!htp]
\begin{center}
\includegraphics{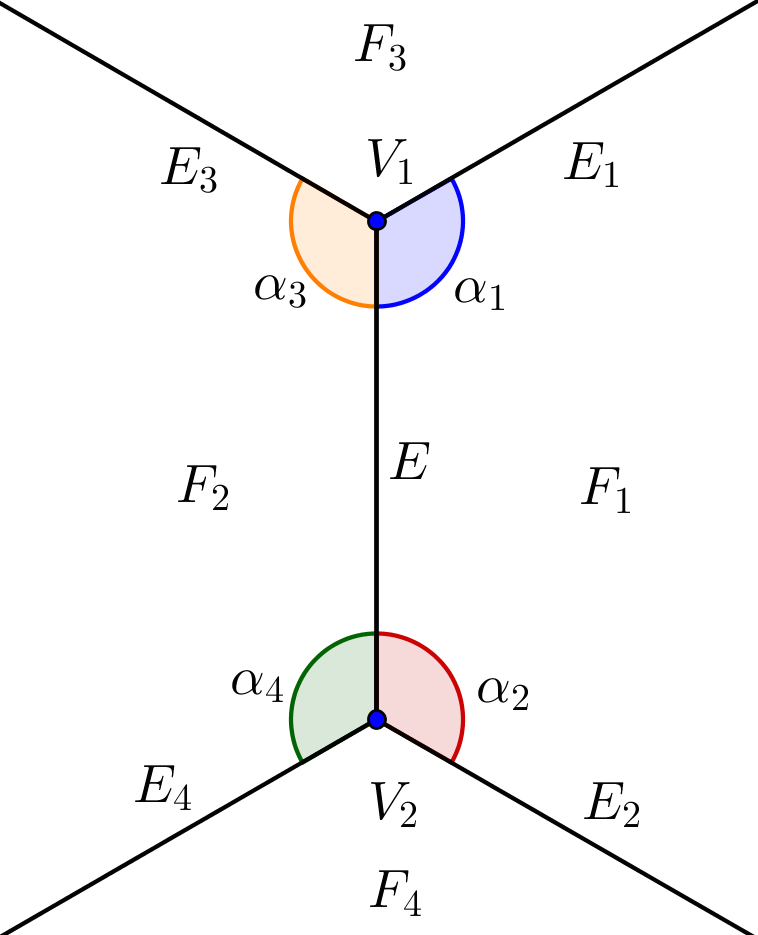}
\caption{The outermost edge}
\label{fig1}
\end{center}
\end{figure}

The next theorem was proved by the author
(Bogachev, 2019, \cite[Theorem 2.1]{Bog19}),
but the formulation is slightly corrected.

\begin{theorem}
\label{th:lenedge}
If $h_J \leqslant h_I$, then the length of the outermost edge satisfies the inequality
$$
a <
\operatorname{arcsinh} \biggl(\frac{\cos ({\alpha_{12}}/{2})}
{\tan ({\alpha_3}/{2})}\biggr)+\operatorname{arcsinh} \biggl(\frac{\cos
({\alpha_{12}}/{2})}{\tan ({\alpha_4}/{2})} \biggr).
$$
\end{theorem}

\begin{proof}
See \cite[Theorem 2.1]{Bog19} and note that
$\tanh(\log(\cot(\alpha_{12}/4))) = \cos(\alpha_{12}/2)$.
 \end{proof}
 
\vskip 0.2cm 

Let us introduce the following notation:
$$
F_{i,j} (\overline{\alpha}) :=
\arcsinh \left(\frac{\cos\left(\frac{\alpha_{12}}{2}\right)}{\tan\left(\frac{\alpha_i}{2}\right)}\right) + \arcsinh \left(\frac{\cos\left(\frac{\alpha_{12}}{2}\right)}{\tan\left(\frac{\alpha_j}{2}\right)}\right).
$$

\begin{corollary}\label{corol:leng}
 The following inequality holds:
$$
\cosh a < \max\{\cosh F_{1,2} (\overline{\alpha}), \cosh F_{3,4} (\overline{\alpha})\}.
$$
\end{corollary}

\subsection{Auxiliary lemmas}

\begin{lemma}
\label{l2.1}
The following relations are true:

{\rm(i)}~$\alpha_{12}+\alpha_{23}+\alpha_{13} > \pi,
\quad \alpha_{12}+\alpha_{24}+\alpha_{14} > \pi$;

{\rm(ii)}
\begin{alignat*}{2}
\cos \alpha_1&=\frac{\cos \alpha_{23}+\cos \alpha_{12} \cdot \cos \alpha_{13}}
{\sin \alpha_{12} \cdot \sin \alpha_{13}}, &\qquad
\cos \alpha_2&=\frac{\cos \alpha_{24}+\cos \alpha_{12} \cdot \cos \alpha_{14}}
{\sin \alpha_{12} \cdot \sin \alpha_{14}},
\\
\cos \alpha_3&=\frac{\cos \alpha_{13}+\cos \alpha_{12} \cdot \cos \alpha_{23}}
{\sin \alpha_{12} \cdot \sin \alpha_{23}}, &\qquad
\cos \alpha_4&=\frac{\cos \alpha_{14}+\cos \alpha_{12} \cdot \cos \alpha_{24}}
{\sin \alpha_{12} \cdot \sin \alpha_{24}}.
\end{alignat*}
\end{lemma}
\begin{proof}
See \cite[Lemma 2.1]{Bog19}.
\end{proof}

\begin{lemma}\label{l2.2}
The following expression for $\cosh F_{i,j} (\overline{\alpha})$ holds:
\begin{small}
$$
\frac{2\cos^2\left(\frac{\alpha_{12}}{2}\right)\cos \left(\frac{\alpha_i}{2}\right) \cos\left(\frac{\alpha_j}{2}\right) +
2\sqrt{\cos^2\left(\frac{\alpha_{12}}{2}\right)\cos^2\left(\frac{\alpha_{i}}{2}\right)+\sin^2\left(\frac{\alpha_i}{2}\right)}
\sqrt{\cos^2\left(\frac{\alpha_{12}}{2}\right)\cos^2\left(\frac{\alpha_{j}}{2}\right)+\sin^2\left(\frac{\alpha_j}{2}\right)}}
{\sqrt{1 - \cos \alpha_i} \sqrt{1-\cos \alpha_j}}.
$$
\end{small}
\end{lemma}
\begin{proof}
Using the formula
$$
\cosh(\arcsinh x + \arcsinh y) = xy + \sqrt{1 + x^2}\sqrt{1+y^2},
$$
we obtain that
\begin{small}
\begin{gather*}
\cosh F_{i,j} (\overline{\alpha})  =
\cosh\left(\arcsinh \left(\frac{\cos\left(\frac{\alpha_{12}}{2}\right)}{\tan\left(\frac{\alpha_i}{2}\right)}\right) + \arcsinh \left(\frac{\cos\left(\frac{\alpha_{12}}{2}\right)}{\tan\left(\frac{\alpha_j}{2}\right)}\right)\right) = \\
  = \frac{\cos^2\left(\frac{\alpha_{12}}{2}\right)\cos \left(\frac{\alpha_i}{2}\right) \cos\left(\frac{\alpha_j}{2}\right) +
\sqrt{\cos^2\left(\frac{\alpha_{12}}{2}\right)\cos^2\left(\frac{\alpha_{i}}{2}\right)+\sin^2\left(\frac{\alpha_i}{2}\right)}
\sqrt{\cos^2\left(\frac{\alpha_{12}}{2}\right)\cos\left(\frac{\alpha_{j}}{2}\right)+\sin^2\left(\frac{\alpha_j}{2}\right)}}{\sin \left(\frac{\alpha_i}{2}\right) \sin\left(\frac{\alpha_j}{2}\right)} = \\
% = \frac{2 \cos^2\left(\frac{\alpha_{12}}{2}\right)\cos \left(\frac{\alpha_i}{2}\right) \cos\left(\frac{\alpha_j}{2}\right) +
%2\sqrt{\cos^2\left(\frac{\alpha_{12}}{2}\right)\cos^2\left(\frac{\alpha_{i}}{2}\right)+\sin^2\left(\frac{\alpha_i}{2}\right)}
%\sqrt{\cos^2\left(\frac{\alpha_{12}}{2}\right)\cos^2\left(\frac{\alpha_{j}}{2}\right)+\sin^2\left(\frac{\alpha_j}{2}\right)}}{\sqrt{2 \sin^2 %\left(\frac{\alpha_i}{2}\right)} \sqrt{2\sin^2\left(\frac{\alpha_j}{2}\right)}} %= \\
 = \frac{2\cos^2\left(\frac{\alpha_{12}}{2}\right)\cos \left(\frac{\alpha_i}{2}\right) \cos\left(\frac{\alpha_j}{2}\right) +
2\sqrt{\cos^2\left(\frac{\alpha_{12}}{2}\right)\cos^2\left(\frac{\alpha_{i}}{2}\right)+\sin^2\left(\frac{\alpha_i}{2}\right)}
\sqrt{\cos^2\left(\frac{\alpha_{12}}{2}\right)\cos^2\left(\frac{\alpha_{j}}{2}\right)+\sin^2\left(\frac{\alpha_j}{2}\right)}}
{\sqrt{1 - \cos \alpha_i} \sqrt{1-\cos \alpha_j}}.
\end{gather*}
\end{small}
Transforming the expressions above, we use the half-angle formulae where appropriate.
\end{proof}

\section{Proof of Theorem~\ref{th1}}
\label{section:th1}

\subsection{Explicit formula for $\ta$.}

Let $E$ be the outermost edge of a compact Coxeter polytope
$P$. Consider the set of unit outer normals  $(u_1, u_2, u_3,
u_4)$ to the facets $F_1$, $F_2$, $F_3$, $F_4$. Note that this vector system is
linearly independent. Its Gram matrix is
$$
G(u_1, u_2, u_3, u_4) =
\begin{pmatrix}
1 & -\cos \alpha_{12} & -\cos \alpha_{13} & -\cos \alpha_{14}\\
-\cos \alpha_{12} & 1 & -\cos \alpha_{23} & -\cos \alpha_{24}\\
-\cos \alpha_{13} & -\cos \alpha_{23} & 1 & -T\\
-\cos \alpha_{14} & -\cos \alpha_{24} & -T & 1
\end{pmatrix},
$$
where $T=|(u_3, u_4)|=\cosh \rho(F_3, F_4)$ is the width of $E$ in the case where the facets $F_3$ and
$F_4$ diverge. Recall that otherwise $T \leqslant  1$, and we do not need to consider this case separately.
Let us denote by $G_{ij}$
the algebraic complements
of the elements of the matrix $G=G(u_1, u_2, u_3, u_4)$.

We denote by $F(\overline \alpha)$ the corresponding
$F_{i,j}(\overline \alpha)$, depending on $h_J \leqslant h_I$ or $h_I \leqslant h_J$ (see Theorem~\ref{th:lenedge}).

\begin{proposition}\label{th:w}
The small ridge $\Sigma_E$ associated with the edge $E$ of the compact Coxeter polytope $P \subset \HH^3$
has width $T$ less than
$$
\ta = \frac{\cosh F(\overline \alpha) \cdot \sqrt{G_{33} G_{44}} - g(\overline \alpha)}{\sin^2 \alpha_{12}},
$$
where
$$
g(\overline \alpha) := \cos \alpha_{12} \cos \alpha_{13} \cos \alpha_{24}
+ \cos \alpha_{12} \cos \alpha_{14} \cos \alpha_{23} + \cos \alpha_{13}
\cos \alpha_{14} + \cos \alpha_{23} \cos \alpha_{24}.
$$
\end{proposition}
\begin{proof}
Let $(u^*_1, u^*_2, u^*_3, u^*_4)$ be the basis dual to the basis $(u_1, u_2, u_3, u_4)$. Then $u^*_3$ and $u^*_4$ determine the vertices $V_2$ and $V_1$ in the Lobachevsky space.
Indeed, the vector $v_1$ corresponding to the point
$V_1 \in \mathbb{H}^3$ is uniquely determined (up to scaling)  by the conditions
$(v_1, u_1)=(v_1, u_2)=(v_1, u_3)=0$. Note that
the vector $u_4^*$ satisfies the same conditions. Therefore, the vectors $v_1$ and
$u^*_4$ are proportional. Hence,
$$
\cosh a=\cosh \rho(V_1, V_2)=-(v_1, v_2)=-\frac{(u_3^*,u_4^*)}
{\sqrt{(u_3^*,u_3^*)(u_4^*,u_4^*)}}.
$$
It is known that $G(u^*_1, u^*_2, u^*_3, u^*_4)=G(u_1, u_2, u_3, u_4)^{-1}$,
whence it follows that $\cosh a$ can be expressed in terms of the algebraic complements
$G_{ij}$ (recall that $G_{ij}$ is computed with the sign $(-1)^{i+j}$) of the elements of the matrix $G=G(u_1, u_2, u_3, u_4)$:
$$
\cosh a=-\frac{(u_3^*,u_4^*)}{\sqrt{(u_3^*,u_3^*)(u_4^*,u_4^*)}} =
\frac{G_{34}}{\sqrt{G_{33} G_{44}}}.
$$
Now 
Theorem~\ref{th:lenedge} implies that $$\cosh
a < \cosh F(\overline \alpha).$$ It follows that
\begin{equation}
\label{eq5}
\frac{G_{34}}{\sqrt{G_{33} G_{44}}} < \cosh F(\overline \alpha).
\end{equation}
For every $\overline \alpha$,  we obtain in this way a linear inequality
with respect to the number $T$.
Indeed,
$$
G_{34} = T(1-\cos^2 \alpha_{12}) + g(\overline \alpha)
= T \cdot \sin^2 \alpha_{12} + g(\overline \alpha) < \cosh F(\overline \alpha) \cdot \sqrt{G_{33} G_{44}},
$$
which completes the proof.
\end{proof}

\subsection{Proof of  Theorem~\ref{th1}}
\label{subsection:th1}

In order to prove  Theorem~\ref{th1} it remains to show that
$$
\max_{\overline{\alpha} \in \Omega} \{\ta\} =
\mathbf{t}_{(\pi/5, \pi/3, \pi/3, \pi/2,
\pi/2)} < 5.75.
$$
Taking into account Lemma~\ref{l2.1},\,(i), 
we can see that only one or two angles $\alpha_{ij}$ can be equal to $\pi/k$, where $k \geqslant 6$. Moreover,
any triple of angles around one of the vertices of the edge $E$ contains $\pi/2$.

\textbf{The plan of the proof.}
Without loss of generality, we can consider separately the following
cases:
\begin{itemize}
    \item[(1)] $\alpha_{12} = \frac{\pi}{k}$, where $k \geqslant 6$. Due to Proposition \ref{prop:a12k}, $\ta < 3$.
    \item[(2)] $\alpha_{13} = \frac{\pi}{k}$, where $k \geqslant 6$. This implies by Lemma~\ref{l2.1},\,(i) that $\alpha_{12} = \pi/2$. By Proposition~\ref{prop:a122}, we have $\ta < 5$.
    \item[(3)] no $\alpha_{ij}$ is equal to $\pi/k$ for $k \geqslant 6$, i.e., $\alpha_{ij} = \pi/2, \pi/3, \pi/4, \pi/5$. This gives us $67$ different possibilities for a small ridge, and $43$ of them are combined by the fact that $\alpha_{12} = \pi/2$. In the latter case, we use Proposition \ref{prop:a122} again: $\ta < 5$.
    \item[(4)] It remains to calculate $\ta$ for $24$ different types of a small ridge with $\alpha_{12} \ne \pi/2$. This is done by 
    using the program {\bf \texttt{SmaRBA}} (\textbf{Sma}ll \textbf{R}idges, \textbf{B}ounds and \textbf{A}pplications, see \cite{SmaRBA}) 
        written in \texttt{Sage} computer algebra system. The result is presented as the list of Coxeter --- Vinberg diagrams in Table
    \ref{tab:ridges}.
\end{itemize}

In order to obtain upper bounds for $\ta$, we shall
use Proposition~\ref{th:w} and the estimate (see Corollary~\ref{corol:leng})
$$
a < F(\overline \alpha) = \max\{F_{1,2}(\overline \alpha), F_{3,4}(\overline \alpha)\}.
$$

\begin{remark}
Note that if one computes bounds for a large (but finite) number of ridges,
 then it can be much more efficient to verify in each case whether $h_J \leqslant h_I$ or $h_I \leqslant h_J$ and depending on that to pick $F_{1,2}(\overline \alpha)$ or $F_{3,4}(\overline \alpha)$.
\end{remark}

\begin{proposition}\label{prop:a12k}
Suppose that $\alpha_{12} = \frac{\pi}{k}$, where $k \geqslant 6$.
Then $\ta < 3$.
\end{proposition}
\begin{proof}
Due to Lemma~\ref{l2.1},\,(i), we can see that all other $\alpha_{ij} = \pi/2$ and
$\overline \alpha = \left(\frac{\pi}{k}, \frac{\pi}{2}, \frac{\pi}{2}, \frac{\pi}{2}, \frac{\pi}{2}\right).$
In this case we have that $F(\overline \alpha) = F_{1,2}(\overline \alpha) = F_{3,4}(\overline \alpha)$
and
\begin{gather*}
    g\left(\frac{\pi}{k}, \frac{\pi}{2}, \frac{\pi}{2}, \frac{\pi}{2}, \frac{\pi}{2}\right)  = 0,\qquad
    \sqrt{G_{33} G_{44}}  = \sin^2 \left(\frac{\pi}{k}\right).
   % F\left(\frac{\pi}{k}, \frac{\pi}{2}, \frac{\pi}{2}, \frac{\pi}{2}, \frac{\pi}{2}\right) & = 2 \arcsinh \left(\frac{\cos (\pi/2k)}{\tan(\pi/4)}\right).
\end{gather*}
We have
$$
\cosh F\left(\frac{\pi}{k}, \frac{\pi}{2}, \frac{\pi}{2}, \frac{\pi}{2}, \frac{\pi}{2}\right) = \cosh\left(2 \arcsinh\left(\frac{\cos (\pi/2k)}{\tan(\pi/4)}\right)\right) =
2\cos^2 \left(\frac{\pi}{2k}\right) + 1.
$$
This implies that
$$
\ta =
\frac{(2\cos^2 \left(\frac{\pi}{2k}\right) + 1) \sin^2 \left(\frac{\pi}{k}\right)}{\sin^2 \left(\frac{\pi}{k}\right)} = 2\cos^2 \left(\frac{\pi}{2k}\right) + 1 < 3.
$$
\end{proof}

Now we can assume that $\alpha_{12} \geqslant \pi/5$. Only one or two angles among remaining $\alpha_{ij}$ can be equal to $\pi/k$, where $k \geqslant 6$.
Without loss of generality, we suppose
that
$\alpha_{13} = \frac{\pi}{k}$, where $k \geqslant 6$. Then $\alpha_{12} = \alpha_{23} = \pi/2$.

If no $\alpha_{ij}$ is equal to $\pi/k$ for $k \geqslant 6$, then
these angles can equal only
$\pi/2$, $\pi/3$, $\pi/4$, and $\pi/5$.
Recall that any triple of angles around one of the vertices of the edge $E$ contains $\pi/2$.

Thus, we can consider separately the case
$\alpha_{12} = \pi/2$.

\begin{proposition}
\label{prop:a122}
If $\alpha_{12} = \pi/2$, then $\ta < 5$.
\end{proposition}
\begin{proof}
We have $\overline{\alpha} = \left(\frac{\pi}{2}, \alpha_{13}, \alpha_{14}, \alpha_{23}, \alpha_{24}\right)$.
Let us now compute:
\begin{equation}
\sqrt{G_{33} G_{44}} = \sqrt{1-\cos^2 \alpha_{13} -\cos^2 \alpha_{23}} \sqrt{1-\cos^2 \alpha_{14} -\cos^2 \alpha_{24}}.
\end{equation}

Notice that (by Lemma~\ref{l2.1})
$$
\begin{aligned}
\cos \alpha_1&=\frac{\cos \alpha_{23}}
{\sin \alpha_{13}}, &\qquad
\cos \alpha_2&=\frac{\cos \alpha_{24}}
{\sin \alpha_{14}},
\\
\cos \alpha_3&=\frac{\cos \alpha_{13}}
{\sin \alpha_{23}}, &\qquad
\cos \alpha_4&=\frac{\cos \alpha_{14}}
{\sin \alpha_{24}}.
\end{aligned}
$$

Using the above expressions and Lemma~\ref{l2.2},
we have
\begin{small}
\begin{gather*}
\ta  \le
\cosh F_{1,2} (\overline{\alpha})
\sqrt{\sin^2 \alpha_{13} -\cos^2 \alpha_{23}} \sqrt{\sin^2 \alpha_{14} -\cos^2 \alpha_{24}} \\
 \leqslant \frac{\cos \left(\frac{\alpha_1}{2}\right) \cos\left(\frac{\alpha_2}{2}\right)}
{\sqrt{1 - \cos \alpha_1} \sqrt{1-\cos \alpha_2}}
\sqrt{\sin^2 \alpha_{13} -\cos^2 \alpha_{23}} \sqrt{\sin^2 \alpha_{14} -\cos^2 \alpha_{24}} \ + \\
 + \frac{
\sqrt{1+\sin^2\left(\frac{\alpha_1}{2}\right)}
\sqrt{1+\sin^2\left(\frac{\alpha_2}{2}\right)}}
{\sqrt{1 - \cos \alpha_1} \sqrt{1-\cos \alpha_2}}
\sqrt{\sin^2 \alpha_{13} -\cos^2 \alpha_{23}} \sqrt{\sin^2 \alpha_{14} -\cos^2 \alpha_{24}} \leqslant \\
 \leqslant \frac{\cos \left(\frac{\alpha_1}{2}\right) \cos\left(\frac{\alpha_2}{2}\right)\sqrt{\sin \alpha_{13} \sin \alpha_{14}}}
{\sqrt{\sin \alpha_{13} -\cos \alpha_{23}} \sqrt{\sin \alpha_{14} -\cos \alpha_{24}}}
\sqrt{\sin^2 \alpha_{13} -\cos^2 \alpha_{23}} \sqrt{\sin^2 \alpha_{14} -\cos^2 \alpha_{24}} \ + \\
 +
\frac{\sqrt{1+\sin^2\left(\frac{\alpha_1}{2}\right)}
\sqrt{1+\sin^2\left(\frac{\alpha_2}{2}\right)}\sqrt{\sin \alpha_{13} \sin \alpha_{14}}}
{\sqrt{\sin \alpha_{13} -\cos \alpha_{23}} \sqrt{\sin \alpha_{14} -\cos \alpha_{24}}}
\sqrt{\sin^2 \alpha_{13} -\cos^2 \alpha_{23}} \sqrt{\sin^2 \alpha_{14} -\cos^2 \alpha_{24}} \leqslant \\
 \leqslant \cos \left(\frac{\alpha_1}{2}\right) \cos\left(\frac{\alpha_2}{2}\right)
\sqrt{\sin \alpha_{13}+\cos \alpha_{23}}
\sqrt{\sin \alpha_{14}+\cos \alpha_{24}} \,\, + \\
 + \sqrt{1+\sin^2\left(\frac{\alpha_1}{2}\right)}
\sqrt{1+\sin^2\left(\frac{\alpha_2}{2}\right)}
\sqrt{\sin \alpha_{13}+\cos \alpha_{23}}
\sqrt{\sin \alpha_{14}+\cos \alpha_{24}} <\\
< \sqrt{2} \sqrt{2} + \frac{\sqrt{3}}{\sqrt{2}}\frac{\sqrt{3}}{\sqrt{2}}\sqrt{2}\sqrt{2} = 5.
\end{gather*}
\end{small}
The last inequality holds, since $\alpha_i \leqslant \pi/2$, and, therefore, $\sin(\alpha_i/2) \leqslant \sin(\pi/4) = 1/\sqrt{2}$.
Absolutely the same argument works for $\ta$ bounded via $\cosh F_{3,4} (\overline{\alpha})$.
\end{proof}

\begin{table}

\begin{tikzpicture}
\draw[thick][fill=black]
(0,0)   circle [radius=.1] % node [left] {$F_4$}

(2,0)   circle [radius=.1]  %node [right] {$F_3$}

(1,1)   circle [radius=.1]  %node [right] {$F_1$}

(1,-1)   circle [radius=.1] %node [left] {$F_2$}
node [below] {$\ta < 2.5$}
;
\draw%[thick]

   %(1.07, 0.92)  -- (1.07,-0.92)
     %(0.92, 0.92)  -- (0.92,-0.92)
     %(1.03, 0.92)  -- (1.03,-0.92)
     (1, 0.9)  -- (1,-0.9) ;  %node [midway,left] {$\mathbf{6}$} ;
     %at ($(0.97, 0.92)!0.25!(0.97,-0.92)$)  {$\mathbf{2}$} ;

 \draw[thick, dashed] (0.1, 0)  -- (1.9,0)
;
\end{tikzpicture}
\quad
\begin{tikzpicture}
\draw[thick][fill=black]
(0,0)   circle [radius=.1]

(2,0)   circle [radius=.1]

(1,1)   circle [radius=.1]

(1,-1)   circle [radius=.1]
node [below] {$\ta < 2.71$};
\draw
   % (0, 0)  -- (1,1)
   % (1, -1) --  (2,0)
    (1.08, 1)  -- (1.08,-1)
     (0.92, 1)  -- (0.92,-1) ;

 \draw[thick, dashed]  (0.1, 0)  -- (1.9,0)
;
\end{tikzpicture}
\quad
\begin{tikzpicture}
\draw[thick][fill=black]
(0,0)   circle [radius=.1]

(2,0)   circle [radius=.1]

(1,1)   circle [radius=.1]

(1,-1)   circle [radius=.1]
node [below] {$\mathbf{t_{\overline{\alpha}}} < 2.87$};
\draw

   (1.09, 0.95)  -- (1.09,-0.95)
     (0.91, 0.95)  -- (0.91,-0.95)
     (1, 1) -- (1, -1)
     ;
     %(1.03, 0.91)  -- (1.03,-0.91)
     %(0.97, 0.91)  -- (0.97,-0.91);

 \draw[thick, dashed]  (0.1, 0)  -- (1.9,0)
;
\end{tikzpicture}
\quad
\begin{tikzpicture}
\draw[thick][fill=black]
(0,0)   circle [radius=.1]

(2,0)   circle [radius=.1]

(1,1)   circle [radius=.1]

(1,-1)   circle [radius=.1]
node [below]
{$\mathbf{t}_{\overline{\alpha}}<3.29$};
\draw
    (1.93, 0.07)  -- (1.07,0.93)

    (1, -1) --  (1,1)

    %(1.04, 1)  -- (1.04,-1)
     %(0.96, 1)  -- (0.96,-1)
     ;

 \draw[thick, dashed]  (0.1, 0)  -- (1.9,0)
;
\end{tikzpicture}
\quad
\begin{tikzpicture}
\draw[thick][fill=black]
(0,0)   circle [radius=.1]

(2,0)   circle [radius=.1]

(1,1)   circle [radius=.1]

(1,-1)   circle [radius=.1]
node [below]
{$\mathbf{t}_{\overline{\alpha}}<3.58$};
\draw
    (1.9, 0)  -- (1,0.9)
    (1.99, 0.09) --  (1.09,1)

    (1, -1) --  (1,1)
    ;

 \draw[thick, dashed]  (0.1, 0)  -- (1.9,0)
;
\end{tikzpicture}
\quad
\begin{tikzpicture}
\draw[thick][fill=black]
(0,0)   circle [radius=.1]
(2,0)   circle [radius=.1]
(1,1)   circle [radius=.1]
(1,-1)   circle [radius=.1]
node [below]
{$\mathbf{t}_{\overline{\alpha}} < 3.72$};
\draw
   % (0, 0)  -- (1,1)
   (1, -1) --  (1,1)

    (1, 1) --  (2,0)
    (1.9, 0)  -- (1, 0.9)
    (2, 0.1) --  (1.08,1.03)
    ;
 \draw[thick, dashed]  (0.1, 0)  -- (1.9,0)
;
\end{tikzpicture}
\vskip 0.7cm

\quad
\begin{tikzpicture}
\draw[fill=black]
(0,0)   circle [radius=.1]

(2,0)   circle [radius=.1]

(1,1)   circle [radius=.1]

(1,-1)   circle [radius=.1]
node [below] {$\ta < 4.07$};
\draw
     (1, -1) --  (1,1)

    (1.93, 0.07)  -- (1.07,0.93)

     (1,1) -- (0,0)
       ;

 \draw[thick, dashed]  (0.1, 0)  -- (1.9,0)
;
\end{tikzpicture}
\quad
\begin{tikzpicture}
\draw[fill=black]
(0,0)   circle [radius=.1]

(2,0)   circle [radius=.1]

(1,1)   circle [radius=.1]

(1,-1)   circle [radius=.1]
node [below] {$\ta < 3.26$} ;
\draw
   (1, -1) --  (1,1)

    (1.93, 0.07)  -- (1.07,0.93)

     (1,-1) -- (0,0) ;

 \draw[thick, dashed]  (0.1, 0)  -- (1.9,0)
;
\end{tikzpicture}
\quad
\begin{tikzpicture}
\draw[fill=black]
(0,0)   circle [radius=.1]

(2,0)   circle [radius=.1]

(1,1)   circle [radius=.1]

(1,-1)   circle [radius=.1]
node [below] {$\ta < 4.62$};
\draw

    (1, -1) --  (1,1)

    (1.9, 0)  -- (1,0.9)
    (1.99, 0.09) --  (1.09,1)

      (0,0.09) -- (0.91,1)
    (0.09,0) -- (1, 0.9)
     ;

 \draw[thick, dashed]  (0.1, 0)  -- (1.9,0)
;
\end{tikzpicture}
\quad
\begin{tikzpicture}
\draw[fill=black]
(0,0)   circle [radius=.1]

(2,0)   circle [radius=.1]

(1,1)   circle [radius=.1]

(1,-1)   circle [radius=.1]
node [below] {$\ta < 3.08$};
\draw
    (1, -1)  -- (1,1)

    (1.9, 0)  -- (1,0.9)
    (1.99, 0.09) --  (1.09,1)

     (0,-0.09) -- (0.91,-1)
    (0.1,0) -- (1, -0.9) ;

 \draw[thick, dashed]  (0.1, 0)  -- (1.9,0)
;
\end{tikzpicture}
\quad
\begin{tikzpicture}
\draw[fill=black]
(0,0)   circle [radius=.1]

(2,0)   circle [radius=.1]

(1,1)   circle [radius=.1]

(1,-1)   circle [radius=.1]
node [below] {$\ta < 4.35$};
\draw

    (1, -1) --  (1,1)

    (1.93, 0.07)  -- (1.07,0.93)

      (0,0.09) -- (0.91,1)
    (0.09,0) -- (1, 0.9)
     ;

 \draw[thick, dashed]  (0.1, 0)  -- (1.9,0)
;
\end{tikzpicture}
\quad
\begin{tikzpicture}
\draw[fill=black]
(0,0)   circle [radius=.1]

(2,0)   circle [radius=.1]

(1,1)   circle [radius=.1]

(1,-1)   circle [radius=.1]
node [below] {$\ta < 3.51$};
\draw
    (1, -1)  -- (1,1)

    (1.93, 0.07)  -- (1.07,0.93)

     (0,-0.09) -- (0.91,-1)
    (0.1,0) -- (1, -0.9) ;

 \draw[thick, dashed]  (0.1, 0)  -- (1.9,0)
;
\end{tikzpicture}
\vskip 0.7cm

\begin{tikzpicture}
\draw[thick][fill=black]
(0,0)   circle [radius=.1]
(2,0)   circle [radius=.1]
(1,1)   circle [radius=.1]
(1,-1)   circle [radius=.1]
node [below]
{$\mathbf{t}_{\overline{\alpha}} < 4.87$};
\draw
   (1, -1) --  (1,1)

    (1, 1) --  (2,0)
    (1.9, 0)  -- (1, 0.9)
    (2, 0.1) --  (1.08,1.03)

     (0, 0) --  (1,1)
     (0.09, 0)  -- (1,0.91)
      (0, 0.09)  -- (1,1.09)
    ;
 \draw[thick, dashed]  (0.1, 0)  -- (1.9,0)
;
\end{tikzpicture}
\quad
\begin{tikzpicture}
\draw[thick][fill=black]
(0,0)   circle [radius=.1]
(2,0)   circle [radius=.1]
(1,1)   circle [radius=.1]
(1,-1)   circle [radius=.1]
node [below]
{$\mathbf{t}_{\overline{\alpha}} < 2.78$};
\draw
   (1, -1) --  (1,1)

    (1, 1) --  (2,0)
    (1.9, 0)  -- (1, 0.9)
    (2, 0.1) --  (1.08,1.03)

     (0, 0) --  (1,-1)
     (0.1, 0)  -- (1,-0.9)
      (0, -0.1)  -- (1,-1.1)
    ;
 \draw[thick, dashed]  (0.1, 0)  -- (1.9,0)
;
\end{tikzpicture}
\quad
\begin{tikzpicture}
\draw[thick][fill=black]
(0,0)   circle [radius=.1]
(2,0)   circle [radius=.1]
(1,1)   circle [radius=.1]
(1,-1)   circle [radius=.1]
node [below]
{$\mathbf{t}_{\overline{\alpha}} < 4.48$};
\draw
   (1, -1) --  (1,1)

    (1, 1) --  (2,0)
    (1.9, 0)  -- (1, 0.9)
    (2, 0.1) --  (1.08,1.03)

     (0, 0) --  (1,1)
    ;
 \draw[thick, dashed]  (0.1, 0)  -- (1.9,0)
;
\end{tikzpicture}
\quad
\begin{tikzpicture}
\draw[thick][fill=black]
(0,0)   circle [radius=.1]
(2,0)   circle [radius=.1]
(1,1)   circle [radius=.1]
(1,-1)   circle [radius=.1]
node [below]
{$\mathbf{t}_{\overline{\alpha}} < 3.62$};
\draw
   (1, -1) --  (1,1)

    (1, 1) --  (2,0)

     (0, 0) --  (1,-1)
     (0.1, 0)  -- (1,-0.9)
      (0, -0.1)  -- (1,-1.1)
    ;
 \draw[thick, dashed]  (0.1, 0)  -- (1.9,0)
;
\end{tikzpicture}
\quad
\begin{tikzpicture}
\draw[thick][fill=black]
(0,0)   circle [radius=.1]
(2,0)   circle [radius=.1]
(1,1)   circle [radius=.1]
(1,-1)   circle [radius=.1]
node [below]
{$\mathbf{t}_{\overline{\alpha}} < 4.75$};
\draw
   (1, -1) --  (1,1)

    (1.9, 0)  -- (1,0.9)
    (1.99, 0.09) --  (1.09,1)

     (0, 0) --  (1,1)
     (0.09, 0)  -- (1,0.91)
      (0, 0.09)  -- (1,1.09)
    ;
 \draw[thick, dashed]  (0.1, 0)  -- (1.9,0)
;
\end{tikzpicture}
\quad\begin{tikzpicture}
\draw[thick][fill=black]
(0,0)   circle [radius=.1]
(2,0)   circle [radius=.1]
(1,1)   circle [radius=.1]
(1,-1)   circle [radius=.1]
node [below]
{$\mathbf{t}_{\overline{\alpha}} < 3.17$};
\draw
   (1, -1) --  (1,1)

    (1.9, 0)  -- (1,0.9)
    (1.99, 0.09) --  (1.09,1)

     (0, 0) --  (1,-1)
     (0.1, 0)  -- (1,-0.9)
      (0, -0.1)  -- (1,-1.1)
    ;
 \draw[thick, dashed]  (0.1, 0)  -- (1.9,0)
;
\end{tikzpicture}
\vskip 0.7cm

\quad
\begin{tikzpicture}
\draw[fill=black]
(0,0)   circle [radius=.1]

(2,0)   circle [radius=.1]

(1,1)   circle [radius=.1]

(1,-1)   circle [radius=.1]
node [below] {$\ta < 3.82$};
\draw
    (1.93, 0.07)  -- (1.07,0.93)

   % (1, -1) --  (2,0)
    (1.08, 1)  -- (1.08,-1)
     (0.92, 1)  -- (0.92,-1) ;

 \draw[thick, dashed]  (0.1, 0)  -- (1.9,0)
;
\end{tikzpicture} \quad
\begin{tikzpicture}
\draw[fill=black]
(0,0)   circle [radius=.1]

(2,0)   circle [radius=.1]

(1,1)   circle [radius=.1]

(1,-1)   circle [radius=.1]
node [below] {$\ta < 4.98$};
\draw
    (0, 0)  -- (1,1)

    (1.93, 0.07)  -- (1.07,0.93)

    (1.08, 1)  -- (1.08,-1)
     (0.92, 1)  -- (0.92,-1) ;

 \draw[thick, dashed]  (0.1, 0)  -- (1.9,0)
;
\end{tikzpicture}\quad
\begin{tikzpicture}
\draw[fill=black]
(0,0)   circle [radius=.1]

(2,0)   circle [radius=.1]

(1,1)   circle [radius=.1]

(1,-1)   circle [radius=.1]
node [below] {$\ta < 4.14$};
\draw
    (1.93, 0.07)  -- (1.07,0.93)

    (1, -1) --  (0,0)

    (1.08, 1)  -- (1.08,-1)
     (0.92, 1)  -- (0.92,-1) ;

 \draw[thick, dashed]  (0.1, 0)  -- (1.9,0)
;
\end{tikzpicture}
\quad
\begin{tikzpicture}
\draw[fill=black]
(0,0)   circle [radius=.1]

(2,0)   circle [radius=.1]

(1,1)   circle [radius=.1]

(1,-1)   circle [radius=.1]
node [below] {$\ta < 4.21$};
\draw
    (1.93, 0.07)  -- (1.07,0.93)

   % (1, -1) --  (2,0)
(1.09, 0.95)  -- (1.09,-0.95)
     (0.91, 0.95)  -- (0.91,-0.95)
     (1, 1) -- (1, -1) ;

 \draw[thick, dashed]  (0.1, 0)  -- (1.9,0)
;
\end{tikzpicture} \quad
\begin{tikzpicture}
\draw[fill=black]
(0,0)   circle [radius=.1]

(2,0)   circle [radius=.1]

(1,1)   circle [radius=.1]

(1,-1)   circle [radius=.1]
node [below] {$\ta < \mathbf{5.75}$};
\draw
    (0, 0)  -- (1,1)

    (1.93, 0.07)  -- (1.07,0.93)

(1.09, 0.95)  -- (1.09,-0.95)
     (0.91, 0.95)  -- (0.91,-0.95)
     (1, 1) -- (1, -1) ;

 \draw[thick, dashed]  (0.1, 0)  -- (1.9,0)
;
\end{tikzpicture}\quad
\begin{tikzpicture}
\draw[fill=black]
(0,0)   circle [radius=.1]

(2,0)   circle [radius=.1]

(1,1)   circle [radius=.1]

(1,-1)   circle [radius=.1]
node [below] {$\ta < 4.9$};
\draw
    (1.93, 0.07)  -- (1.07,0.93)

    (1, -1) --  (0,0)

(1.09, 0.95)  -- (1.09,-0.95)
     (0.91, 0.95)  -- (0.91,-0.95)
     (1, 1) -- (1, -1) ;

 \draw[thick, dashed]  (0.1, 0)  -- (1.9,0)
;
\end{tikzpicture}
\vskip 1cm
\caption{Coxeter --- Vinberg diagrams of the remaining small ridges}
\label{tab:ridges}
\end{table}

\begin{figure}
\begin{center}
\begin{tikzpicture}
\draw[fill=black]
(0,0)   circle [radius=.1]  node [left] {$F_4$}

(2,0)   circle [radius=.1]  node [right] {$F_3$}

(1,1)   circle [radius=.1]  node [above] {$F_1$}

(1,-1)   circle [radius=.1] node [below] {$F_2$}

;
\draw

   (1.09, 0.95)  -- (1.09,-0.95)
     (0.91, 0.95)  -- (0.91,-0.95)
     (1, 1) -- (1, -1) ;

 \draw[thick, dashed] (0.1, 0)  -- (1.9,0)
;
\end{tikzpicture}
\vskip 0.07cm
$\mathbf{t_{\overline{\alpha}}} < 2.87$
\end{center}
\caption{Coxeter --- Vinberg diagram of a small ridge $(\pi/5, \pi/2, \pi/2, \pi/2, \pi/2)$.}
\label{fig:ridge-example}
\end{figure}
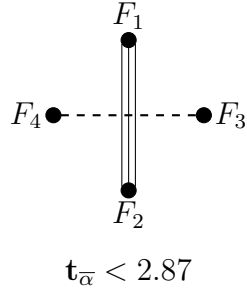

\vskip 0.5cm

After that, it remains to calculate $\ta$ for $24$ different types $\overline{\alpha}$ of small ridges with~$\alpha_{12}\ne\pi/2$.
This is done by using the program {\bf \texttt{SmaRBA}} \cite{SmaRBA} written in the computer algebra system \texttt{Sage}. The result is presented in Table \ref{tab:ridges} as the list of Coxeter --- Vinberg diagrams for 
the facets $F_1$, $F_2$, $F_3$, $F_4$. The facets $F_3$ and $F_4$ are connected by a dotted line, and the whole diagram is signed by the relevant bound: $\ta < \mathrm{constant}$.
The numbering of the facets of each diagram is the same as in Fig.~\ref{fig:ridge-example},
which shows an example of how the ridge diagram looks like when
$\overline \alpha = (\pi/5, \pi/2, \pi/2, \pi/2, \pi/2)$. We see from this picture that $\mathbf{t}_{(\pi/5, \pi/2, \pi/2, \pi/2, \pi/2)} < 2.87$.

The numbers given in Table~\ref{tab:ridges} are calculated by {\bf \texttt{SmaRBA}} \cite{SmaRBA} with accuracy up to eight decimal places. We show them rounded up to the nearest hundredth, which is quite enough for our purposes. For example, the maximal found number approximately equals
$5.74850431686$, which is rounded up to $5.75$.

Thus, combining Propositions \ref{prop:a12k}, \ref{prop:a122}, and the list in Table~\ref{tab:ridges}, we obtain
the proof of  Theorem~\ref{th1}. \qed

\subsection{Proof of Corollary~\ref{corol:th1}}
\label{subsect:corolth1}
Let $P$ be a compact Coxeter polytope in $\HH^{n \geqslant 4}$. Suppose that $P'$ is a $3$-dimensional face of $P$ that
 is itself a Coxeter polytope. Let
$O$ be an interior point of~$P'$ and let $E \in P'$ be the outermost edge from this point.

Then $P'$ has ($2$-dimensional) facets $F_1$ and $F_2$, framing the edge $E$, and, by  Theorem~\ref{th1},
$\cosh \rho(F_1, F_2) \leqslant \ta < 5.75$. Recall that the compact hyperbolic Coxeter polytope $P$ is simple. This implies that
$F_1$ and $F_2$ belong to the facets $P_1$ and
$P_2$ of $P$, respectively, where $P_1$ and $P_2$
are also the framing facets for the edge $E$.
Then we have
$$
\cosh \rho(P_1, P_2) \leqslant \cosh \rho(F_1, F_2) \leqslant \ta < 5.75.
$$
\qed

\section{Proof of  Theorem \ref{th2}}
\label{section:th2}

%\subsection{Auxiliary results}

The distance from the point $e_0 \in \HH^n$, where $(e_0, e_0) = -1$,
to the plane
$$H_{u_1, \ldots, u_k} := \{x \in \HH^n \mid x \in \langle u_1,
\ldots, u_k \rangle^{\perp}, \ (u_j, u_j) = 1, \ 1\leqslant j \leqslant k\}$$
can be calculated by the formula
\begin{equation}\label{eq:rasst}
\sinh^2 \rho(e_0,  H_{u_1, \ldots, u_k}) =\sum_{i, j}\overline{g_{ij}} y_i y_j,
\end{equation}
where $\overline{g_{ij}}$ are the elements of the inverse matrix
$G^{-1} = G(u_1, \ldots, u_k)^{-1}$, and
$$y_j = - (e_0, u_j) = -\sinh \rho (e_0, H_j), \quad H_j := H_{u_j} = \{x \mid (x,u_j) = 0\}$$ for all $1 \leqslant j \leqslant k$ (we assume that $(e_0, u_j) \leqslant 0$, i.e., $e_0 \in H^-_{u_j}$).

Let $P$ be a compact Coxeter polytope in $\HH^n$ whose small ridge $\Sigma_E$
(associated with the outermost edge
$E$ from some interior point $O \in P$ given by the vector
$e_0 \in \R^{n,1}$ such that $(e_0, e_0) = -1$) is
right-angled. Let $F_1, \ldots, F_{n-1}$ be the
facets of $P$ containing $E$ with unit outer normals
$u_1, \ldots, u_{n-1}$, and let $u_n$ and $u_{n+1}$
be
the unit outer normals to the framing
facets $F_n$ and $F_{n+1}$
containing the vertices of $E$ but not $E$ itself.

Let us consider the following Gram matrix
$$
G(e_0, u_1, u_2, \ldots, u_{n+1}) =
\begin{pmatrix}
-1 & -y_1 & \ldots & -y_n & -y_{n+1}\\
-y_1 & 1 & 0 & 0 & 0\\
\vdots & \vdots & \ddots & \vdots & \vdots\\
-y_n & 0  & 0 & 1 & -T\\
-y_{n+1} & 0 & 0 & -T & 1
\end{pmatrix},
$$
where
$$
y_j = -(e_0, u_j) = - \sinh \rho (e_0, H_j), \quad H_j = \{x \mid (x,u_j) = 0\}.
$$

We can assume that $y_n \leqslant y_{n+1}$. The fact that the edge $E$ is the outermost edge from the point $O$ gives us the inequalities
$\rho(O, E) \geqslant \rho(O, E')$ for any edge $E'$ adjacent to $E$. Recall that $E' \in H_n$ or $E' \in H_{n+1}$. Assume that $E' \in H_n$ and $E' \not\in H_j$ for some $j \leqslant n-1$. By (\ref{eq:rasst}), the distances from the point $O$ to edges $E$ and $E'$ of $P$ satisfy the following:
$$
\sinh^2 \rho (O, E) = y_1^2 + \ldots + y_{n-1}^2, \quad
\sinh^2 \rho (O, E') = y_1^2 + \ldots + y_{j-1}^2 + y_{j+1}^2 + \ldots + y_{n-1}^2 + y_n^2.
$$
Applying the above consideration to every edge $E'$ adjacent to $E$, we obtain that
\begin{equation}\label{right-ineq}
y_n \leqslant y_{n+1} \leqslant y_1, y_2, \ldots, y_{n-1}, \qquad (n-1) y^2_{n+1} \leqslant y_1^2 + \ldots + y_{n-1}^2.    
\end{equation}
Since $n+2$ vectors
$e_0, u_1, \ldots, u_{n+1}$
belong to the $(n+1)$-dimensional vector space $\R^{n,1}$,
then
$$
\det G(e_0, u_1, \ldots, u_{n+1}) = (y_1^2 + \ldots + y_{n-1}^2 + 1) T^2 - 2 y_n y_{n+1} T
- (y_1^2 + \ldots +
y_{n+1}^2 + 1) = 0,
$$
i.e.,
$$
T = \dfrac{y_n y_{n+1} + \sqrt{y_n^2 y_{n+1}^2 + AB}}{A},
$$
where
$$
A := y_1^2 + \ldots + y_{n-1}^2 + 1,\quad
B := y_1^2 + \ldots + y_{n+1}^2 + 1.
$$
Therefore (from (\ref{right-ineq})),
$$
2 y_n y_{n+1} \leqslant y_n^2 + y_{n+1}^2 \le
2 y_{n+1}^2 < \frac{2A}{n-1}, \qquad \frac{B}{A} = 1 + \frac{y_n^2 + y^2_{n+1}}{A} < 1 + \frac{2}{n-1} = \frac{n+1}{n-1}
$$
and
$$
T = \frac{y_n y_{n+1}}{A} + \sqrt{\left(\frac{y_n y_{n+1}}{A}\right)^2 + \frac{B}{A}} < \frac{1}{n-1} + \sqrt{\frac{1}{(n-1)^2} + \frac{n+1}{n-1}} = \frac{n+1}{n-1}.
$$
\qed

\section{Arithmetic hyperbolic reflection groups and reflective Lorentzian lattices}
\label{sec:arith}

\subsection{Definitions and preliminaries}

Suppose that $\mathbb{F}$ is a totally real algebraic number field with the ring of integers
$A = \mathcal{O}_{\mathbb{F}}$.

\begin{definition}
A free finitely generated $A$-module $L$ with an inner product of signature $(n,1)$
is said to be a Lorentzian lattice if,
 for each non-identity embedding $\sigma \colon \mathbb{F} \to \mathbb{R}$,
  the quadratic space $L \otimes_{\sigma(A)} \mathbb{R}$ is positive definite.
\end{definition}

Suppose that $L$ is a Lorentzian lattice.
It is naturally embedded in the $(n + 1)$-dimensional real  \emph{Minkowski space} $\R^{n,1} = L \otimes_{\id(A)} \R$. We take one of the connected components of the hyperboloid
\begin{equation}\label{eq:hyp}
\{v \in \R^{n,1} \mid (v, v) = -1\}
\end{equation}
as a vector model of the $n$-dimensional hyperbolic Lobachevsky space $\HH^n$.

Suppose that $\OOO(L)$ is the group of automorphisms of a lattice $L$. It is known
(cf. \cite{Ven37, BHC62, MT62})
that its subgroup
$\mathcal{O}'(L)$ leaving invariant each connected component of the hyperboloid~(\ref{eq:hyp}),
is a discrete group of motions of the Lobachevsky space with finite volume fundamental polytope. Moreover,
if $\F = \Q$ and the lattice $L$ is isotropic (that is, the quadratic form  associated with it represents zero), then
the quotient space $\HH^n/\Gamma$
is a finite volume non-compact orbifold,
and in all other cases it is compact. 

Recall that two subgroups $\Gamma_1$ and $\Gamma_2$ of some group $G$ are said to be \textit{commensurable} if, for some element $g \in G$, the group $\Gamma_1 \cap g \Gamma_2 g^{-1}$ is a finite index subgroup in each of them.

\begin{definition}
Discrete subgroups $\Gamma < \Isom(\HH^n)$ that are commensurable
with $\OOO'(L)$ are called arithmetic lattices of the simplest type. The field $\F$ is called the field of definition (or the ground field) of the
group $\Gamma$ (and of all subgroups commensurable with it).
\end{definition}

A primitive vector
$e$ of a Lorentzian lattice $L$ is called a \textit{root} or, more precisely,
a $k$-\textit{root},
where $k = (e, e) \in A_{>0}$ if
$2 (e, x) \in kA$ for all $x \in L.$
Every root $e$ defines an \emph{orthogonal reflection} (called
a \emph{$k$-reflection} if
$(e, e) = k$) in the space $L \otimes_{\id(A)} \mathbb{R}$
$$
\mathcal {R}_e: x \mapsto x - \frac{2 (e, x)} {(e, e)} e,
$$
which preserves the lattice $L$ and determines
the reflection of the space $\mathbb{H}^n$ with respect to the hyperplane
$H_e = \{x\in\mathbb{H}^n \mid (x,e) = 0\},$
called the \textit{mirror} of $\mathcal{R}_e$.

\begin{definition}\label{def:sub-2}
A reflection $\mathcal{R}_e$ is said to be a sub-$2$-reflection if
$(e, e) \mid 2$ in $A$.
\end{definition}

For example, for $\F = \Q$ and $A = \Z$, this holds for $(e, e) = 1$ and $(e, e) = 2$, i.e., only $1$- and $2$-reflections
are sub-$2$-reflections, while, for $\F = \Q(\!\sqrt{2})$ and $A = \Z[\!\sqrt{2}]$, all $1$-,
$2$- and $(2 + \sqrt{2})$-reflections  are
sub-$2$-reflections. Any primitive vector $e \in L$ for which $(e, e) \mid 2$ is automatically a root of the lattice $L$ and of any of its finite extensions.

Let $L$ be a Lorentzian lattice over a ring of integers $A$. We denote by $\OOO_r (L)$ the subgroup of the group $\OOO'(L)$ generated by all
reflections contained in it, and
we denote by $\mathcal{S}(L)$ the subgroup of $\mathcal{O}'(L)$ generated by all sub-$2$-reflections.

\begin{definition}\label{def:refl-lat}
A Lorentzian lattice $L$ is said to be reflective if the index $ [\OOO'(L): \OOO_r(L)]$ is finite, and sub-$2$-reflective if the index  $[\OOO' (L): \mathcal{S}(L)]$ is finite.
\end{definition}

Vinberg showed \cite[Prop. 3]{Vin72} that $\OOO'(L) = \OOO_r(L) \rtimes H,$ where $H = \Sym(P) \cap \OOO'(L)$ and $P$ is the fundamental Coxeter polytope of the arithmetic hyperbolic reflection group $\OOO_r(L)$. Thus, a lattice $L$ is reflective if and only if $P$ has a finite volume.

\begin{definition}
A Lorentzian $\Z$-lattice $L$
is called $2$-reflective if the subgroup
$\OOO^{(2)}_r(L)$ generated by all $2$-reflections
has a finite index in $\OOO'(L)$.
\end{definition}

Note that any $2$-reflective lattice is sub-$2$-reflective.
Obviously, a finite extension of any sub-$2$-reflective Lorentzian lattice is also a sub-$2$-reflective Lorentzian lattice.

\begin{remark}
Due to the mentioned property about finite extensions, sub-$2$-reflective lattices can be also called stably reflective.

In \cite{Bog16,Bog17,Bog19} sub-$2$-reflective lattices over $\Z$ are called $(1{,}2)$-reflective. 
\end{remark}

\subsection{State of the art}

As mentioned in the introduction, Vinberg
\cite{Vin67}
began in 1967  a systematic study of hyperbolic reflection groups. He
proved an arithmeticity criterion for finite covolume hyperbolic reflection groups and,
in particular, he showed that a
discrete hyperbolic reflection group of finite covolume is an arithmetic group with
ground field $\F$
if and only if it is commensurable with a
group of the form $\OOO'(L)$, where $L$ is
some (automatically reflective)
Lorentzian lattice over a totally real number
field $\F$.

In~1972, Vinberg proposed an algorithm
(see \cite{Vin72}, \cite{Vin73}) that, given a lattice $L$, enables
one to construct the fundamental Coxeter polytope of the group $\mathcal{O}_r (L)$ and determine thereby
the reflectivity of the lattice $L$.

The next important result belongs to several authors.

\begin{theorem}[see \cite{Vin84, Nik81a,  LMR06, Ag06, Nik07, ABSW08}]
For each $n \geqslant 2$, up to scaling, there are
only finitely many reflective Lorentzian
lattices of signature $(n,1)$. Similarly,  up to conjugacy, there
are only finitely many maximal arithmetic
reflection groups in the spaces $\HH^n$.
Arithmetic hyperbolic reflection groups and
compact Coxeter polytopes do not exist
in $\HH^{n\geqslant 30}$.
\end{theorem}

Note that Allcock \cite{All06} constructed
infinitely many hyperbolic reflection groups
including arithmetic ones. Thus, the 
maximality assumption in the above theorem
cannot be dropped.

It was also proved that there are no reflective hyperbolic $\Z$-lattices of rank
$n+1 > 22$ and for $n = 20$ (Esselmann, 1996 \cite{Ess96}).

The above results give the hope that all reflective Lorentzian lattices, as well as
maximal arithmetic hyperbolic reflection groups,
can be classified.

Here we describe some progress in the
problem of classification of reflective
Lorentzian lattices. A more detailed history of the problem can be found in the recent survey of Belolipetsky \cite{Bel16}.

For the ground field $\Q$: reflective Lorentzian
lattices of signature $(n,1)$ are classified for $n = 2$ (Nikulin, 2000 \cite{Nik00}, and Allcock, 2011 \cite{Al12}),
$n = 4$ (Scharlau and Walhorn, 1989--1993 \cite{SW92, Wal93}), $n = 5$ (Turkalj, 2017 \cite{Tur17,Tur18})
and in the non-compact (isotropic) case for $n = 3$ (Scharlau and Walhorn, 1989--1993 \cite{Sch89, SW92}).

A classification of reflective Lorentzian
lattices of signature $(2,1)$ over
$\Z[\!\sqrt {2}]$ was obtained by Mark in
2015 \cite{Mark15b,Mark15phd}.

Unimodular reflective Lorentzian
lattices over $\Z$ were classified by Vinberg
and Kaplinskaja, (1972 and 1978, see
\cite{Vin72, Vin72a, VK78}). Other classifications of unimodular reflective Lorentzian lattices over
$\Z[\!\sqrt{2}]$, $\Z[(1+\sqrt{5})/2]$ and $\Z[\cos(2\pi/7)]$,
were obtained by Bugaenko (1984, 1990 and 1992, see \cite{Bug84, Bug90, Bug92}).

In 1979, 1981, and 1984 (see \cite{Nik79, Nik81a, Nik84}), Nikulin obtained a classification of
$2$-reflective Lorentzian $\Z$-lattices of signature
$(n,1)$ for $n \ne 3$, and Vinberg classified these
lattices for $n=3$ (1998 and 2007, \cite{Vin98, Vin07}).
Finally, the author obtained
(see \cite{Bog16,Bog17,Bog19})
a classification of sub-$2$-reflective anisotropic Lorentzian $\Z$-lattices of signature $(3,1)$.
They all turned out to be $2$-reflective in this case.

In all other cases, the classification problem still
remains open.

\subsection{Methods of testing a lattice for sub-$2$-reflectivity and non-reflectivity}

Recall that there is Vinberg's algorithm that constructs the fundamental Coxeter polytope of the group $\mathcal{O}_r (L)$. It can be applied to a group of type $\mathcal{S}(L)$. However, it will be more efficient to apply the procedure of Vinberg's algorithm to the large group  $\mathcal{O}_r (L)$ and  use some other approach to determine whether $L$
is sub-$2$-reflective or not.

\subsubsection{Method of ``bad'' reflections}
\label{ss4.2}

If we can construct the fundamental Coxeter polyhedron (or some part of it) of the group $\OOO_r (L)$
for some Lorentzian lattice $L$, then it is possible to determine whether it is sub-$2$-reflective.
One can consider the group $\Delta$ generated by the $k$-reflections that are not sub-$2$-reflections (we shall call them
``bad'' reflections)
in the sides of the fundamental polyhedron of the group $\OOO_r (L)$.
The following lemma holds (see~\cite{Vin07}).

\begin{lemma}
\label{lem:badrefl}
A lattice $L$ is sub-$2$-reflective if and only if it is reflective and the group
$\Delta$ is finite.
\end{lemma}

Actually, to prove that a lattice is not sub-$2$-reflective,
it is sufficient to construct only some part of the fundamental polyhedron containing
an infinite subgroup generated by bad reflections.

\subsubsection{Method of infinite symmetry}
\label{ssec:infsym}

Recall that $$\OOO'(L) = \OOO_r(L) \rtimes H,$$ where $H = \Sym(P) \cap \OOO'(L)$. If $P$ is of
infinite volume and has infinitely many facets, then the group $H$ is infinite. To determine whether it
is infinite or not, one can use the following lemma proved by V.\,O.~Bugaenko in 1992 (see \cite[Lemma 3.1]{Bug92}).

\begin{lemma}
Suppose $H$ is a discrete subgroup of $\Isom(\HH^n)$. Then
$H$ is infinite if and only if there exists a subgroup of $H$ without fixed points in $\HH^n$.
\end{lemma}

How can we find the set of fixed points?

\begin{lemma}[Bugaenko, see Lemma 3.2 in \cite{Bug92}]
Let $\eta$ be an involutive transformation of a real vector space $V$. Then the set of its fixed
points $Fix(\eta)$ is generated by the vectors $e_j + \eta(e_j)$, where $\{e_j\}$ form a basis of $V$.
\end{lemma}

Due to this lemma the proof of non-reflectivity of a lattice is the following. If we know some part
of a polyhedron $P$ for the group
$\OOO_r(L)$, then
we can find a few symmetries of its Coxeter --- Vinberg diagram.

If these symmetries preserve the lattice
$L$, then they generate the subgroup that preserves $P$. If this subgroup has no fixed
points, then $\OOO_r(L)$ is of infinite index
in $\OOO'(L)$.

\section{Classification of sub-$2$-reflective Lorentzian lattices of signature $(3,1)$}
\label{section:appl}

\subsection{Description of the method}
In this section we describe application of  Theorem~\ref{th1} or, more precisely, of Proposition~\ref{th:w} to classification
of sub-$2$-reflective Lorentzian lattices.

Let now $P$ be the fundamental Coxeter polytope of the group $\mathcal{S}(L)$ for an anisotropic Lorentzian lattice $L$
of signature $(3,1)$ over a ring of integers
$\OOO_F$ of any totally real number field $\F$. The lattice $L$ is sub-$2$-reflective if and only if the polytope
$P$ is compact (i.\,e., bounded) in~$\HH^3$.

Let $E$~be an edge (of the polytope $P$) corresponding to a small ridge of width not greater than $\ta$. By  Theorem~\ref{th1} we can ensure that $\ta < 5.75$, however, we shall use a more efficient way, an explicit
formula from Proposition~\ref{th:w}.

Indeed, for a fixed number field $\F$ only finitely many dihedral angles in Coxeter polytopes are possible. This leaves us only
finitely many combinatorial types of a small ridge, and for each such type one can
explicitly compute (see {\bf \texttt{SmaRBA}} \cite{SmaRBA}) the respective bound
$\ta$. We present some useful calculations in Table
\ref{tab:ground_fields} (possible values for $(u,u)$ can be computed using \cite[Thm 9.29]{NZM}). 

\begin{table}[]
    \centering
    \begin{tabular}{c|c|c|c|c}
        $\F$ & Possible values for $(u,u)$ & Possible angles & \# of different ridges & $\max \ta$ \\
        \hline
        $\Q$ & $1$, $2$ & $\frac{\pi}{2}$, $\frac{\pi}{3}$, $\frac{\pi}{4}$, $\frac{\pi}{6}$ & 44 &$4.98$ \\
        $\Q(\!\sqrt{2})$ & $1$, $2$, $2 + \sqrt{2}$ & $\frac{\pi}{2}$,  $\frac{\pi}{3}$, $\frac{\pi}{4}$, $\frac{\pi}{6}$, $\frac{\pi}{8}$ & 58 & $4.98$ \\
        $\Q(\!\sqrt{3})$ & $1$, $2$, $2 + \sqrt{3}$, $4 + 2\sqrt{3}$ & $\frac{\pi}{2}$,  $\frac{\pi}{3}$, $\frac{\pi}{4}$, $\frac{\pi}{6}$, $\frac{\pi}{12}$ & 58 & $4.98$ \\
        $\Q(\!\sqrt{5})$ & $1$, $2$ & $\frac{\pi}{2}$,  $\frac{\pi}{3}$, $\frac{\pi}{4}$, $\frac{\pi}{5}$, $\frac{\pi}{6}$, $\frac{\pi}{10}$ & 99 & $5.75$ \\
    \end{tabular}
    \vskip 0.5cm
    \caption{Some quantities for sub-$2$-reflective Lorentzian lattices over ground fields $\F = \Q, \Q(\!\sqrt{2}), \Q(\!\sqrt{3}), \Q(\!\sqrt{5})$, i.e. $\OOO_\F = \Z, \Z[\!\sqrt{2}], \Z[\!\sqrt{3}], \Z[\frac{1+\sqrt{5}}{2}]$.}
    \label{tab:ground_fields}
\end{table}

\begin{remark}
Due to a minor technical error, the bound
$\ta < 4.14$ in \cite[Theorem 1.1]{Bog19}
is incorrect (the correct one is $\ta < 4.98$). However, the result \cite[Theorem 1.2]{Bog19} still holds.
\end{remark}

Let $u_1$, $u_2$
be the roots of the lattice $L$ that are orthogonal to the facets containing
the edge $E$ and are the outer normals of these facets.
Similarly, let $u_3$, $u_4$ be the roots corresponding to the framing
facets. We denote these facets
by $F_1$, $F_2$, $F_3$, and $F_4$, respectively. If $(u_3, u_3) = k$, $(u_4, u_4) = l$, then
(by  Theorem~\ref{th1})
\begin{equation}
\label{eq2}
|(u_3, u_4)| \leqslant  5.75 \sqrt{kl}.
\end{equation}

Since we are solving the classification problem for sub-$2$-reflective lattices,
we consider those roots $u \in L$ that satisfy the condition
$(u,u) \mid 2$ in $\OOO_F$. Thus, $(u,u)$ always assumes finitely many values (see Table~\ref{tab:ground_fields}).

In this case we are given bounds
on all elements of the matrix $G(u_1, u_2, u_3, u_4)$,
because all the facets $F_i$ are pairwise intersecting, excepting,
possibly, the pair of faces $F_3$ and $F_4$. But if they do not intersect,
then the distance between
these faces is bounded by inequality~\eqref{eq2}.
Thus, all entries of the matrix $G(u_1, u_2, u_3, u_4)$ are integers and bounded, so
there are only finitely many possible matrices
$G(u_1, u_2, u_3, u_4)$.

The vectors $u_1, u_2, u_3, u_4$ generate some sublattice $L '$
of finite index in the lattice $L$. More precisely, the lattice $ L $ lies between
the lattices $L'$ and $(L')^*$, and
$$
[(L')^* : L']^2=|d(L')|.
$$
Hence it follows that $|d(L')|$ is divisible by $[L:L']^2$. Using this, in each case we shall
find  for a lattice
$L'$ all its possible extensions of finite index.

The resulting list of candidate lattices is verified for reflectivity using Vinberg’s algorithm. There exist a few software implementations of Vinberg's algorithms, these are {\bf \texttt{AlVin}} \cite{Guglielmetti2, Gug17}, for Lorentzian lattices with an orthogonal basis over several ground fields, and {\bf \texttt{VinAl}} (see \cite{VinAlg2017, BP18})
for Lorentzian lattices with an arbitrary basis over $\Z$.
Further work on
the project that implements Vinberg's algorithm for arbitrary lattices over the quadratic fields
$\Q(\!\sqrt{d})$ is being carried out jointly with A.\,Yu.~Perepechko.

We introduce some notation:

\begin{itemize}
\item[1)]~$[C]$ is a quadratic lattice whose inner product in some
basis is given by a symmetric matrix $C$;

\item[2)]~$d(L) := \det C$ is the discriminant of the lattice $L = [C]$;

\item[3)]~$L \oplus M$ is the orthogonal sum of the lattices $L$ and $M$.
%\item[4)]~$[k]L$ is the quadratic lattice obtained from $L$ by multiplying all inner products by $k \in \mathbb{Z}$;
\end{itemize}

The method described above, allows one to obtain the following fact.

\begin{theorem}[Bogachev, \cite{Bog19},
Th.~1.2]
Every sub-$2$-reflective anisotropic Lorentzian lattice of signature $(3,1)$ over $\mathbb{Z}$ is either isomorphic
to
$[-7] \oplus [1] \oplus [1] \oplus [1]$ or
$[-15] \oplus [1] \oplus [1] \oplus [1]$,
or to an even index $2$ sublattice of one of them.
\end{theorem}

Using {\bf \texttt{AlVin}}, one can easy get the following result.

\begin{theorem}
Unimodular Lorentzian lattices over $\Q(\!\sqrt{13})$ and
$\Q(\!\sqrt{17})$ of signature $(n,1)$ are reflective if and only if $n \leqslant 4$ (see Table~\ref{tab:unimod} for details).
\end{theorem}

\begin{table}
\caption{Unimodular reflective Lorentzian lattices over $\Q(\!\sqrt{13})$ and $\Q(\!\sqrt{17})$.}
\label{tab:unimod}
\begin{tabular} {| r | c | c | c | c |}
\hline
$L\qquad\quad$ & $n$ & $\#$ facets \\
\hline
$[-\frac{3+\sqrt{13}}{2}] \oplus [1] \oplus \ldots \oplus [1]$ & $2$ &  $4$   \\
\hline
$[-\frac{3+\sqrt{13}}{2}] \oplus [1] \oplus \ldots \oplus [1]$& $3$ &  $9$ \\
\hline
$[-\frac{3+\sqrt{13}}{2}] \oplus [1] \oplus \ldots \oplus [1]$& $4$ &  $40$ \\
\hline
\end{tabular}\label{tab:listuni13}
\begin{tabular} {| r | c | c | c | c |}
\hline
$L\qquad\quad$ & $n$ & $\#$ facets  \\
\hline
$[-4-\sqrt{17}] \oplus [1] \oplus \ldots \oplus [1]$ & $2$ &  $4$   \\
\hline
$[-4-\sqrt{17}] \oplus [1] \oplus \ldots \oplus [1]$& $3$ &  $6$ \\
\hline
$[-4-\sqrt{17}] \oplus [1] \oplus \ldots \oplus [1]$& $4$ &  $20$ \\
\hline
\end{tabular}\label{tab:listuni17}    
\end{table}

The next step is to find a short list of
candidate-lattices for sub-$2$-reflectivity over $\Z[\!\sqrt{2}]$.

\subsection{Short list of candidate-lattices}

Our program {\bf \texttt{SmaRBA}} \cite{SmaRBA} creates a list of numbers $\ta$ (with respect to the ground field $\Q(\!\sqrt{2})$, see Table~\ref{tab:ground_fields}) and
then, using
this list, displays all Gram matrices $G(u_1, u_2, u_3, u_4)$.

This list consists of $83$ matrices, but many of them
are pairwise isomorphic.
After reducing this list, we obtain matrices  $G_1$--$G_{12}$, for each of which we find all corresponding
extensions.

To each Gram matrix $G_k$ in our notation, there corresponds a lattice $L_k$ that can have some
other extensions. For each new lattice
(non-isomorphic to any previously found lattice) we introduce the notation
$L(k)$, where $k$ denotes its number:

\vskip 0.3cm
\begin{small}
$
G_{1} =
\begin{pmatrix}
1 & 0 & 0 & 0\\
0 & 1& 0 & 0\\
0 & 0 & 1 & -1-\sqrt{2}\\
0 & 0 &  -1-\sqrt{2} & 1
\end{pmatrix},  \qquad L_1 \simeq [ -2(1+\sqrt{2})] \oplus [1] \oplus [1] \oplus [1];
$
\end{small}

Its unique extension is the ``index $\sqrt{2}$'' extension
$$L(1) := [ -(1+\sqrt{2})] \oplus [1] \oplus [1] \oplus [1].$$

\vskip 0.3cm

\begin{small}
$
G_{2} =
\begin{pmatrix}
1 & 0 & 0 & 0\\
0 & 1& 0 & 0\\
0 & 0 & 1 & -1-\sqrt{2}\\
0 & 0 &  -1-\sqrt{2} & 2
\end{pmatrix},  \qquad L_2 \simeq [ -(1+2\sqrt{2})] \oplus [1] \oplus [1] \oplus [1] := L(2);
$\end{small}
\vskip 0.3cm

\begin{small}
$
G_{3} =
\begin{pmatrix}
1 & 0 & 0 & 0\\
0 & 1& 0 & 0\\
0 & 0 & 2 & -1-\sqrt{2}\\
0 & 0 &  -1-\sqrt{2} & 2
\end{pmatrix},  \qquad L_3 \simeq  [1-2\sqrt{2}] \oplus [1] \oplus [1] \oplus [1] := L(3);
$\end{small}
\vskip 0.3cm
Notice that $L(3) \simeq [-5-4\sqrt{2}] \oplus [1] \oplus [1] \oplus [1]$, since $-(5+4\sqrt{2}) = (1-2\sqrt{2})(1+\sqrt{2})^2$.

\vskip 0.3cm
%\begin{small}
%$
%G_{4} =
%\begin{pmatrix}
%1 & 0 & 0 & 0\\
%0 & 1& 0 & 0\\
%0 & 0 & 2 & -1-2\sqrt{2}\\
%0 & 0 &  -1-2\sqrt{2} & 2
%\end{pmatrix},  \ L_4 \simeq %L(3);
%$
%\end{small}
%\vskip 0.3cm

%\begin{small}
%$
%G_{5} =
%\begin{pmatrix}
%1 & 0 & 0 & 0\\
%0 & 1& -1 & -1\\
%0 & -1 & 2 & -1-\sqrt{2}\\
%0 & -1 &  -1-\sqrt{2} & 2
%\end{pmatrix},  \ L_5 \simeq %[-5-4\sqrt{2}] \oplus [1] %\oplus [1] \oplus [1] \simeq %L(3);
%$\end{small}
%\vskip 0.3cm

\begin{small}
$
G_{4} =
\begin{pmatrix}
1 & 0 & 0 & 0\\
0 & 1& -1 & -1\\
0 & -1 & 2 & -1-2\sqrt{2}\\
0 & -1 &  -1-2\sqrt{2} & 2
\end{pmatrix},  \ L_4 \simeq [-11-8\sqrt{2}] \oplus [1] \oplus [1] \oplus [1]  := L(4);
$\end{small}
\vskip 0.3 cm

\begin{small}
$
G_{5} =
\begin{pmatrix}
1 & 0 & 0 & 0\\
0 & 1& 0 & 0\\
0 & 0 & 2 & -2-2\sqrt{2}\\
0 & 0 &  -2-2\sqrt{2} & 2
\end{pmatrix} , \ L_5  = [G_5];
$\end{small}

\vskip 0.3cm

Its unique extension is the ``index $\sqrt{2}$'' extension
$$L(5) := [-\sqrt{2}] \oplus [1] \oplus [1] \oplus [1].$$

\vskip 0.3cm

\begin{small}
$
G_{6} =
\begin{pmatrix}
1 & 0 & 0 & 0\\
0 & 2& -1 & -1\\
0 & -1 & 2 & -\sqrt{2}\\
0 & -1 &  -\sqrt{2} & 2
\end{pmatrix},  \ L_6 = [G_6];
$\end{small}
\vskip 0.3cm

Its unique extension is the ``index $\sqrt{2}$'' extension
\begin{small}
$$
\begin{bmatrix}2 & -1 &-\sqrt{2}\\-1 & 2 &\sqrt{2}-1 \\ -\sqrt{2} & \sqrt{2}-1 & 2-\sqrt{2} \end{bmatrix}\oplus [1] \simeq L(5).
$$\end{small}
\vskip 0.3cm

\begin{small}
$
G_{7} =
\begin{pmatrix}
1 & 0 & 0 & 0\\
0 & 2& -1 & -1\\
0 & -1 & 2 & -\sqrt{2}-1\\
0 & -1 &  -\sqrt{2}-1 & 2
\end{pmatrix},  \ L_7 = [G_7];
$ \end{small}
\vskip 0.3cm

Its unique extension is the ``index $\sqrt{2}$'' extension
\begin{small}
$$
L(6) := \begin{bmatrix}2 & -1 &0 \\-1 & 2 &-1 \\ 0 & -1 & -\sqrt{2} \end{bmatrix}\oplus [1] .
$$\end{small}
\vskip 0.25cm

$
G_{8} =
\begin{pmatrix}
1 & 0 & 0 & 0\\
0 & 2& -1 & -1\\
0 & -1 & 2 & -\sqrt{2}-2\\
0 & -1 &  -\sqrt{2}-2 & 2
\end{pmatrix},  \ L_{8} = [G_{8}];
$

\vskip 0.25cm

Its unique extension is the ``index $\sqrt{2}$'' extension \begin{small}
$$L(7) := \begin{bmatrix}2 & -1-\sqrt{2}\\-1-\sqrt{2} & 2\end{bmatrix} \oplus [2+\sqrt{2}]
\oplus [1] \simeq [-5-4\sqrt{2}] \oplus [2+\sqrt{2}] \oplus [1] \oplus [1] .$$\end{small}

\vskip 0.25cm

\begin{small}
$
G_{9} =
\begin{pmatrix}
1 & 0 & 0 & -1\\
0 & 2& -1 & -1\\
0 & -1 & 2 & -\sqrt{2}-1\\
-1 & -1 &  -\sqrt{2}-1 & 2
\end{pmatrix},  \ L_{9} \simeq [-7-6\sqrt{2}] \oplus [1] \oplus [1] \oplus [1] := L(8);
$
\vskip 0.25cm

$
G_{10} =
\begin{pmatrix}
1 & 0 & -1 & -1\\
0 & 2& -1 & -1\\
-1 & -1 & 2 & -2\sqrt{2}-1\\
-1 & -1 &  -2\sqrt{2}-1 & 2
\end{pmatrix},  \ L_{10} = [G_{10}];
$\end{small}
\vskip 0.25cm

Its unique extension is the index $2$ extension
$$L(9) := [-7-5\sqrt{2}] \oplus [1] \oplus [1] \oplus [1].$$
%\vskip 0.25 cm

\begin{small}
$
G_{11} =
\begin{pmatrix}
2 & 0 & 0 & -1\\
0 & 2& -1 & -1\\
0 & -1 & 2 & -\sqrt{2}\\
-1 & -1 &  -\sqrt{2} & 2
\end{pmatrix} , \ L_{11} = [G_{11}] := L(10);
$
\vskip 0.2cm
$
G_{12} =
\begin{pmatrix}
2 & 0 & 0 & -1\\
0 & 2& -1 & -1\\
0 & -1 & 2 & -\sqrt{2}-1\\
-1 & -1 &  -\sqrt{2}-1 & 2
\end{pmatrix},  \ L_{12} = [G_{12}] := L(11).
$
\end{small}

\section{Sub-$2$-reflectivity test and proof of  Theorem~\ref{th3}}\label{section:th3}

So far we have $11$ candidate lattices $L(1)$--$L(11)$. For each lattice
$L(k)$ we will use Vinberg's algorithm for constructing the fundamental Coxeter 
polytope for the group $\OOO_r(L(k))$. After that, it remains to apply Lemma~\ref{lem:badrefl}.

\begin{table}
\begin{footnotesize}
\begin{center}
\begin{tabular}{c|c|c|c}
$L(k)$ & $L$ & Roots &
$B(L)$
\\
\hline
$L(1)$ & $[-1 - \sqrt{2}] \oplus [1] \oplus [1] \oplus [1]$ &
\begin{minipage}{0.49\textwidth}
\centering
$$
\begin{aligned}
a_1 &= (0, -1, 1, 0) &
a_2 &= \left(0, 0, -1, 1\right)
\\
a_3 &= \left(0, 0, 0, -1\right) &
a_4 &= \left(1, 1+\sqrt{2}, 0, 0\right)
\end{aligned}
$$
$a_5 = \left(1+\sqrt{2}, 1+\sqrt{2}, 1+\sqrt{2}, 1+\sqrt{2}\right)$
\end{minipage}
& $\emptyset$ \\
\hline
$L(2)$ & $[-1 - 2\sqrt{2}] \oplus [1] \oplus [1] \oplus [1]$ &
\begin{minipage}{0.49\textwidth}
\centering
$$
\begin{aligned}
a_1 &= (0, -1, 1, 0) &
a_2 &= \left(0, 0, -1, 1\right)
\\
a_3 &= \left(0, 0, 0, -1\right) &
a_4 &= \left(1, 1+\sqrt{2}, 0, 0\right)
\end{aligned}
$$
$a_5 = \left(1+\sqrt{2}, 2+\sqrt{2}, 2+\sqrt{2}, 1\right)$
$a_6 = \left(1+\sqrt{2}, 2+\sqrt{2}, 1+\sqrt{2}, 1+\sqrt{2}\right)$
\end{minipage}
& $\emptyset$\\
\hline
$L(3)$ & $[-5 - 4\sqrt{2}] \oplus [1] \oplus [1] \oplus [1]$ &
\begin{minipage}{0.49\textwidth}
\centering
$$
\begin{aligned}
a_1 &= (0, -1, 1, 0) &
a_2 &= \left(0, 0, -1, 1\right)
\\
a_3 &= \left(0, 0, 0, -1\right) &
a_4 &= \left(1, 3+\sqrt{2}, 0, 0\right)
\end{aligned}
$$
$a_5 = \left(1, 1+\sqrt{2}, 1+\sqrt{2}, 1\right)$
\end{minipage}
& $a_4$\\
\hline
$L(4)$ & $[-11 - 8\sqrt{2}] \oplus [1] \oplus [1] \oplus [1]$ &
\begin{minipage}{0.49\textwidth}
\centering
$$
\begin{aligned}
a_1 &= (0, -1, 1, 0) &
a_2 &= \left(0, 0, -1, 1\right)
\\
a_3 &= \left(0, 0, 0, -1\right) &
a_4 &= \left(1, 2+\sqrt{2}, 2+\sqrt{2}, 1\right)
\end{aligned}
$$
$a_5 = \left(1, 2+2\cdot\sqrt{2}, 1, 0\right)$

$a_6 = \left(1, 2+\sqrt{2}, 1+\sqrt{2}, 1+\sqrt{2}\right)$

$a_7 = \left(2+\sqrt{2}, 7+5\sqrt{2}, 3+3\sqrt{2}, 2+\sqrt{2}\right)$

$a_8 = \left(1+2\sqrt{2}, 8+5\sqrt{2}, 4+3\sqrt{2}, 3+2\sqrt{2}\right)$

$a_9 = \left(1+2\sqrt{2}, 8+6\sqrt{2}, 3+2\sqrt{2}, 2+2\sqrt{2}\right)$

$a_{10} = \left(2+3\sqrt{2}, 13+9\sqrt{2}, 7+5\sqrt{2}, 2+\sqrt{2}\right)$

$a_{11} = \left(4+2\sqrt{2}, 13+10\sqrt{2}, 9+6\sqrt{2}, 0\right)$

$a_{12} = \left(4+4\sqrt{2}, 19+14\sqrt{2}, 9+6\sqrt{2}, 8+6\sqrt{2}\right)$
	
$a_{13} = \left(4+4\sqrt{2}, 20+14\sqrt{2}, 11+8\sqrt{2}, 1\right)$

$a_{14} = \left(4+2\sqrt{2}, 14+10\sqrt{2}, 6+4\sqrt{2}, 5+4\sqrt{2}\right)$

$a_{15} = \left(4+3\sqrt{2}, 17+12\sqrt{2}, 8+5\sqrt{2}, 6+4\sqrt{2}\right)$

$a_{16} = \left(4+3\sqrt{2}, 17+12\sqrt{2}, 9+7\sqrt{2}, 1+\sqrt{2}\right)$

$a_{17} = \left(5+4\sqrt{2}, 22+15\sqrt{2}, 13+9\sqrt{2}, 1+\sqrt{2}\right)$

\end{minipage}
& $\emptyset$\\
\hline
$L(5)$ & $[-\sqrt{2}] \oplus [1] \oplus [1] \oplus [1]$ &
\begin{minipage}{0.49\textwidth}
\centering
$$
\begin{aligned}
a_1 &= (0, -1, 1, 0) &
a_2 &= \left(0, 0, -1, 1\right)
\\
a_3 &= \left(0, 0, 0, -1\right) &
a_4 &= \left(1+\sqrt{2}, 1+\sqrt{2}, 1+\sqrt{2}, 0\right)
\end{aligned}
$$
$a_5 = \left(1+\sqrt{2}, 2+\sqrt{2}, 0, 0\right)$

$a_6 = \left(2+\sqrt{2}, 1+\sqrt{2}, 1+\sqrt{2}, 1+\sqrt{2}\right)$
\end{minipage}
& $\emptyset$\\
\hline
$L(9)$ & $[-7 - 5\sqrt{2}] \oplus [1] \oplus [1] \oplus [1]$ &
\begin{minipage}{0.49\textwidth}
\centering
$$
\begin{aligned}
a_1 &= (0, -1, 1, 0) &
a_2 &= \left(0, 0, -1, 1\right)
\\
a_3 &= \left(0, 0, 0, -1\right) &
a_4 &= \left(2-\sqrt{2}, 1+\sqrt{2}, 1, 0\right)
\end{aligned}
$$
$a_5 =  \left(1, 1+\sqrt{2}, 1+\sqrt{2}, 1+\sqrt{2}\right)$
\end{minipage}
& $\emptyset$\\
\hline
$L(10)$ & $\begin{bmatrix}
2 & 0 & 0 & -1\\
0 & 2& -1 & -1\\
0 & -1 & 2 & -\sqrt{2}\\
-1 & -1 &  -\sqrt{2} & 2
\end{bmatrix}$ & 
\begin{minipage}{0.49\textwidth}
\centering
$$
\begin{aligned}
a_1 &= (-1, 0, 0, 0) &
a_2 &= \left(0, -1, 0,0\right)
\\
a_3 &= \left(0, 0, -1, 0\right) &
a_4 &= \left(1, 1+\sqrt{2}, 1+\sqrt{2}, 1\right)
\end{aligned}
$$
$a_5 =  \left(2+\sqrt{2}, 2+\sqrt{2}, 2+\sqrt{2}, 1+\sqrt{2}\right)$
\end{minipage} & $\emptyset$  \\
\hline
\end{tabular}
\end{center}
\end{footnotesize}
    \caption{Sub-$2$-reflective Lorentzian
    lattices of signature $(3,1)$ over $\Z[\!\sqrt{2}]$. Here $B(L)$ denotes the set of ``bad'' reflections.}
    \label{tab:stably-details}
\end{table}

%\subsection{Lattices with an orthogonal basis}

For candidate lattices with an orthogonal basis we use the software implementation
{\bf \texttt{AlVin}} \cite{Guglielmetti2} of Vinberg's algorithm.
This program is written for Lorentzian lattices associated with diagonal quadratic forms with square-free coefficients.

For lattices with non-orthogonal basis
we use the program {\bf \texttt{VinAlNF}}
\cite{VinAlNF2021} written by R. Bottinelli in 2020. The project {\bf \texttt{VinAlNF}} is actually just the modified program {\bf \texttt{VinAl}} \cite{VinAlg2017}, rewritten in \texttt{Julia} for arbitrary lattices over number fields.

As the result, we obtain seven sub-$2$-reflective Lorentzian lattices of signature
$(3,1)$ over $\Z[\!\sqrt{2}]$, which are represented in Table~\ref{tab:stably-details}.

The Gram matrices and Coxeter --- Vinberg diagrams corresponding to all the
lattices $L(1)$--$L(12)$ can be obtained by
{\bf \texttt{SmaRBA}} \cite{SmaRBA}, but for sub-$2$-reflective ones their Coxeter --- Vinberg diagrams are depicted in Fig.~\ref{fig:thC}.

We shall prove that all remaining lattices
are not sub-$2$-reflective (some of them are reflective, but not sub-$2$-reflective).

\begin{proposition}
The lattice
$L(6) = \begin{footnotesize}
\begin{bmatrix}2 & -1 &0 \\-1 & 2 &-1 \\ 0 & -1 & -\sqrt{2} \end{bmatrix}
\end{footnotesize}\oplus [1]$
is not sub-$2$-reflective.
\end{proposition}
\begin{proof}
The program {\bf \texttt{VinAlNF}}
finds $9$ first roots

\begin{small}
\begin{gather*}
\begin{aligned}
a_1 & = \left(-1,\,0,\,0,\,0\right),&
a_2 & = \left(0,\,-1,\,0,\,0\right), & a_3 & = \left(0,\,0,\,0,\,-1\right),\\
a_4 & = \left(\sqrt{2}, 1, 2+\sqrt{2},0\right), &
a_5 & = \left(1,1,1,\sqrt{2}\right),&
a_6 & = \left(1,2,1,0\right),\\
a_7 & = \left(\sqrt{2},1+\sqrt{2},\sqrt{2},1\right),&
a_8 & = \left(2 + 2\sqrt{2}, 3+3\sqrt{2}, 2+2\sqrt{2}, 2+2\sqrt{2}\right), & a_9 &= \left(1,2,3, 3+2\sqrt{2}\right).
\end{aligned}
\end{gather*}
\end{small}

Let us consider the subgroup generated by
``bad'' reflections with respect to the mirrors $H_{a_4}$ and $H_{a_9}$.
Since the mirrors  $H_{a_4}$ and $H_{a_9}$ are divergent, this subgroup is infinite.
\end{proof}

\begin{proposition}
The lattice $L(7) =
[-5-4\sqrt{2}] \oplus [1] \oplus [1]  \oplus [2+\sqrt{2}]$
is reflective, but not sub-$2$-reflective.
\end{proposition}
\begin{proof}
The program {\bf \texttt{AlVin}}
finds $20$ roots, but we need only first $7$ of them:
\begin{small}
\begin{gather*}
\begin{aligned}
a_1 & = \left(0, -1, 1, 0\right),&
a_2 & = \left(0, 0, -1, 0\right), & a_3 & = \left(0, 0, 0, -1\right),\\
a_4 & = \left(1, 3+\sqrt{2}, 0,0\right), &
a_5 & = \left(1,1+\sqrt{2},0,1+\sqrt{2}\right),&
a_6 & = \left(1,1+\sqrt{2},1+\sqrt{2},0\right),\\
\end{aligned}
a_7 = \left(1+\sqrt{2},0,0,3+\sqrt{2}\right).
\end{gather*}
\end{small}

It is sufficient to consider the group generated by
``bad'' reflections with respect to the mirrors $H_{a_4}$ and $H_{a_7}$.
Since these mirrors are divergent, this subgroup is infinite.
\end{proof}

\begin{proposition}
The lattice $L(8) = [ -7-6\sqrt{2}] \oplus [1] \oplus [1] \oplus [1]$ is reflective, but not
sub-$2$-reflective.
\end{proposition}
\begin{proof}
For the lattice $L(8)$ the program
{\bf \texttt{AlVin}} \cite{Guglielmetti2}
found 10 roots:
\begin{small}
\begin{gather*}
\begin{aligned}
a_1 &= \left(0,\,-1,\,1,\,0\right),
& a_2 &= \left(0,\,0,\,-1,\,1\right),\\
a_3 &= \left(0,\,0,\,0,\,-1\right),
& a_4 &= \left(1,\,\sqrt{2}+ 1,\,\sqrt{2}+ 1,\,\sqrt{2}+ 1\right),\\
a_5 &= \left(1,\,\sqrt{2}+ 2,\,\sqrt{2}+ 1,\,0\right), &
a_6 &= \left(2 \sqrt{2}+ 1,\,6 \sqrt{2}+ 7,\,0,\,0\right),\\
a_7 &=  \left(\sqrt{2}+ 1,\,3 \sqrt{2}+ 5,\,\sqrt{2}+ 1,\,1\right), &
a_8 &= \left(\sqrt{2}+ 1,\,3 \sqrt{2}+ 4,\,\sqrt{2}+ 2,\,\sqrt{2}+ 2\right),
\\
a_9 & = \left(4 \sqrt{2}+ 6,\,13 \sqrt{2}+ 19,\,7 \sqrt{2}+ 12,\,6 \sqrt{2}+ 7\right),
&
a_{10} & =  \left(2 \sqrt{2}+ 2,\,6 \sqrt{2}+ 9,\,2 \sqrt{2}+ 3,\,2 \sqrt{2}+ 2\right).
\end{aligned}
\end{gather*}
\end{small}
The Gram matrix of this set of roots
 corresponds to a compact $3$-dimensional Coxeter polytope.
The main diagonal of this matrix equals
$\{2, 2, 1, 2, 2, 2 \sqrt{2}+ 10, 2, 1, 2 \sqrt{2}+ 10\}.$
It remains to see that the group generated
by ``bad'' reflections with respect to the mirrors
$H_{a_6}$ and $H_{a_{10}}$ is infinite,
since the respective vertices of the Coxeter --- Vinberg diagram are connected by the dotted edge.
Hence the lattice $L(8)$ is
reflective, but not sub-$2$-reflective.
\end{proof}

\begin{proposition}
The lattice $L(11)$ is not reflective.
\end{proposition}
\begin{proof}
The non-reflectivity of this lattice is determined by the method of infinite order symmetry
described in \S \ref{ssec:infsym}. This method was implemented in {\bf \texttt{VinAlNF}} \cite{VinAlNF2021}.
\end{proof}

\vskip 0.5cm

Thus, the lattices $L(1)$--$L(5)$, $L(9)$, $L(10)$ are sub-$2$-reflective. 
This completes the proof of  Theorem~\ref{th3}. \qed

\vskip 0.2cm

The next step in this direction can be finding all
sub-$2$-reflective Lorentzian
$\Z[\!\sqrt{2}]$-lattices of signature $(3,1)$ or
over other rings of integers. Moreover, the author hopes that the geometric method of small ridges can be generalized to higher dimensions.

\end{document}